\documentclass[a4paper,11pt]{article}
\usepackage[utf8]{inputenc}
\usepackage{amssymb,amsmath,mathrsfs,amsthm,bbm}
\usepackage{verbatim}%for comment
\usepackage{indentfirst} % for indent
\usepackage{geometry}
\usepackage{bbm}% for mathbb 1

\usepackage{hyperref} %navigation
\hypersetup{colorlinks=true,allcolors=blue}
\usepackage{hypcap}

\newtheorem{thm}{Theorem}[section]
\newtheorem{defi}[thm]{Definition}
\newtheorem{prop}[thm]{Proposition}

\newtheorem{lem}[thm]{Lemma}
\newtheorem{rem}[thm]{Remark}

\newcounter{cnstcnt}

\numberwithin{equation}{section}

\renewcommand{\rm}[1]{\mathrm{#1}}
\renewcommand{\cal}[1]{\mathcal{#1}}

\newcommand{\bb}[1]{\mathbb{#1}}

\def\RR{\bb R}
\def\BB{\mathbbm 1}
\def\R{\bb R}
\def\C{\bb C}

\newcommand{\dd}{\mathrm{d}}

\newcommand{\supp}{\rm{supp}}

\newcommand{\murr}[1]{(\mu_{#1}*\mu_{#1}^-)^{(r)}}

\def\fren {\tau}

\def\ndelta {\cal N_\delta}
\def\nrho {\cal N_\rho}
\def\l{\langle }
\def\r{\rangle }
\def\id{\rm{id}}

\def\nude{\nu_\delta}

\begin{document}
	\bibliographystyle{alpha}
	\title{\textbf{Discretized Sum-product and Fourier decay in $\bb R^n$}}
	\author{Jialun LI}
	\date{}
	\maketitle
\begin{abstract}
	We generalize Bourgain's theorem on the decay of the Fourier transform of the multiplicative convolution of measures on $\R$ to the ring $\R^n$, where the multiplication is given by coordinate multiplication.
\end{abstract}	
	
%\tableofcontents
	
\section{Introduction}

The purpose of this manuscript is to generalize a result of Bourgain in \cite{bourgain2010discretized} to $\RR^n$. This result deals with the Fourier decay of the multiplicative convolution of Borel probability measures on $\RR$. 

If $E$ is a metric space, we write $B_E(x,r)$ for a closed ball centred at $x$ of radius $r$. Vectors in $\R^n$ are seen as column vectors. The product structure on $\R^n$ is given by coordinate, that is, for $x=(x^1,\cdots, x^n)$ and $y=(y^1,\cdots, y^n)$ in $\R^n$ the product is defined to be $xy=(x^1y^1,\cdots, x^ny^n)$. We equip $\R^n$ with the standard scalar product $\langle\cdot,\cdot \rangle$. For a Borel probability measure on $\R^n$, let $\mu_k$ be the $k$-times multiplicative convolution of $\mu$. For a vector $v$ in the unit sphere $\bb S^{n-1}$, let $\pi_v$ be the projection from $\R^n$ to $\R$ given by $\pi_v(x)=\langle v, x\rangle$ for $x$ in $\R^n$.

\begin{thm}\label{thm:sumfourier}
	Given $\kappa_0>0$, there exist $\epsilon,\epsilon_1>0$ and $k\in \bb N$ such that the following holds for $\delta>0$ small enough. Let $\mu$ be a probability measure on $[1/2,1]^n\subset \R^n$ which satisfies $(\delta,\kappa_0,\epsilon)$ projective non concentration assumption, that is, 
	\begin{equation}\label{equ:noncon}
	\forall\rho\geq\delta,\quad\sup_{a\in\bb R,v\in\bb S^{n-1}}(\pi_v)_*\mu(B_\RR(a,\rho))=\sup_{a,v}\mu\{x| \l v, x\r\in B_\RR(a,\rho) \}\leq\delta^{-\epsilon} \rho^{\kappa_0}.
	\end{equation}	
	
	Then for all $\xi\in\bb R^n$ with $\|\xi\|\in[\delta^{-1}/2,\delta^{-1}]$,
	\begin{equation}\label{equ:exp2pi}
	|\hat{\mu}_k(\xi)|=\left|\int \exp(2i\pi\l \xi, x_1\cdots x_k\r)\dd\mu(x_1)\cdots \dd\mu(x_k)\right|\leq \delta^{\epsilon_1}.
	\end{equation}
	\end{thm}
	\begin{rem}
		This theorem will be used in \cite{li2018spectralgap} to give a Fourier decay of the Furstenberg measures for algebraic real split Lie groups, similar to Bourgain's result been used in the work of Bourgain and Dyatlov \cite{bourgain2017fourier} for Patterson-Sullivan measures associated to convex-cocompact Fuchsian groups. 
		
		From the Fourier decay of the Furstenberg measures, we can establish an exponential error term in the renewal theorem in the context of products of random matrices, which will also be used in \cite{LS} to obtain a power decay of the Fourier transform of some self-affine measures.
	\end{rem}
	\begin{rem}
		We cannot have a sharper result like $\epsilon_1\geq n/2$, because here we only use the product structure. In $\RR^*$, there exist Borel subgroups which have fractional dimension. (See \cite{erdos1966additive} for example) For a measure supported on a fractional Borel subgroup, the decay rate of Fourier transform is controlled by the Hausdorff dimension of the Borel subgroup. Hence, fractional Borel subgroups are obstacles for large decay rate of Fourier transform.
		
		If we continue to exploit the additive structure, that is to say replacing $\mu_k$ by $\nu=(\mu_k)^{*r}$, the $r$-times additive convolution of $\mu_k$, then the Fourier transform of $\nu$ can have arbitrary large decay rate. 
		
		The Fourier transform detects the additive structure. But our measure $\mu_k$ has the multiplicative structure. The decay of Fourier transform means that the additive and multiplicative structures are hard to coexist, the sum-product philosophy.
	\end{rem}
	The projective non concentration means the projection of the measure $\mu$ on every one dimensional linear subspace $\bb Rv$ satisfies a non concentration assumption (the case of $\RR$).
	
	The case $n=1$ is due to Bourgain \cite[Lemma 8.43]{bourgain2010discretized}. The main ingredient of the proof of Fourier decay is the discretized sum-product estimates in $\bb R^n$. The sum-product estimate roughly says that if the set does not concentrate in small balls, then under addition or multiplication the size of the set will become robustly larger than the initial set. 
	
	For $\delta>0$ and a bounded set $A$ in a metric space $E$, let $\ndelta(A)$ be the minimal number of closed balls of radius $\delta$ needed to cover $A$. In a metric space, we say that a set $A$ is \textbf{$\rho$ away from} a set $B$ if $A$ is not contained in the $\rho$ neighbourhood of $B$, that is, there exists $x$ in $A$ such that $d(x,B)\geq \rho$. In $(\R^*)^n$, we use $\id$ to denote the identity element $(1,\cdots ,1)$. In $\R^n$, we will consider maximal proper unitary subalgebras, such subalgebras are given by $\{x\in\R^n| x^i=x^j  \}$ for $1\leq i<j\leq n$. We say that $A$ is $\rho$ away from proper unitary subalgebras of $\R^n$ if $A$ is $\rho$ away from any maximal proper unitary subalgebra of $\R^n$.
	
	Now we state the discretized sum-product estimates on $\R^n$, which is the main ingredient of the proof of Theorem \ref{thm:sumfourier}.
\begin{thm}\label{prop:sumpro}
	We will consider the action of $(\bb R^*)^n$ on $V=\bb R^n$. The action is given by $gv=(g^1v^1,\dots,g^nv^n)$ for $g$ in $(\R^*)^n$ and $v$ in $V$. There exists a neighbourhood $U$ of the identity in $(\R^*)^n$ such that the following holds. Given $\kappa>0,\sigma\in(0,n)$, there exists $\epsilon>0$ such that for all $\delta>0$ sufficiently small, if $A\subset U$ and $X\subset B_V(0,\delta^{-\epsilon})$ satisfy the following $(\delta,\kappa,\sigma,\epsilon)$ assumption:

	(i)For  $j=1,\dots, n$
	\begin{equation*}
	\forall\rho\geq\delta,\ \nrho(\pi_j(A))\geq\delta^\epsilon\rho^{-\kappa},
	\end{equation*} 
	where $\pi_j$ denotes the projection into $j$-th coordinate,
	
	(ii) $A$ is $\delta^\epsilon$ away from proper unitary subalgebras of $\bb R^n$,
	
	(iii)For $j=1,\dots, n$
	\begin{equation*}
	\forall\rho\geq\delta,\ \nrho(\pi_j(X))\geq\delta^\epsilon\rho^{-\kappa},
	\end{equation*}
	
	(iv)$\ndelta(X)\leq \delta^{-(n-\sigma)-\epsilon}$.
	
	Then
	\begin{equation*}
	\ndelta(X+X)+\sup_{a\in A}\ndelta(X+aX)\geq \delta^{-\epsilon}\ndelta(X).
	\end{equation*}
\end{thm}
\begin{rem}
	The case $n=1$ is due to Bourgain \cite{bourgain2010discretized}. Compared with \cite[Prop.1]{bourgain_spectral_2012}, our situation does have invariant subspace under the action. Hence we put more regularity on the projection into coordinate subspaces.

	This theorem relies on a recent result of He and de Saxcé \cite{he_sum-product_2018}. See Proposition \ref{prop:sumrep}.
\end{rem}
\begin{rem}
	Roughly speaking, (i) and (iii) mean that the projections of $A,X$ into coordinate subspaces are not concentrated. Assumption (ii) is reasonable since it prevents $A$ from being trapped in a subalgebra.
	
	Compared with the projective non concentration in Theorem \ref{thm:sumfourier}, the assumption here is weaker. In multiplicative convolution, we need additionally that $\mu$ is not trapped in any affine subspace.
\end{rem}
%From the discretized sum-product theorem to the Fourier decay of multiplicative convolution can be found in \cite{bourgain2010discretized}. The analogue result for finite fields is established in \cite{bourgain_estimates_2006}. See also \cite{green_sum-product_2009}, where the author gave a really clear treatment of the sum-product phenomenon in $\bb F_p$. The proof of Theorem \ref{thm:sumfourier} from Theorem \ref{prop:sumpro} will be given in Section \ref{sec:appmul}.
%
\subsection*{Notation}
We will make use of some classic notation: For two real valued functions $A$ and $B$, we write $A=O(B), A\ll B$ or $B\gg A$ if there exists constant $C>0$ such that $|A|\leq CB$, where $C$ only depends on the ambient space. We also write $A\sim B$ if $B\ll A\ll B$. 

We write $A=O_r(B), A\ll_r B$, $B\gg_r A$ and $A\sim_r B$ if the constant $C$ depends on an extra parameter $r>0$.
%We write $\|\cdot \|$ for the norm on a normed algebra.

\section{Discretized sum-product estimates in $\bb R^n$}
The non concentration assumption in Theorem \ref{thm:sumfourier} is a little different from that in \cite{bourgain2010discretized}, but the two assumptions are equivalent up to constants.
\begin{lem}\label{lem:equivalent}
		Let $1>\delta>0$. Let $\nu$ be a Borel probability measure on $\R$. We have two non concentration assumptions. 
		\begin{itemize}
			\item[(1)]$(\delta,\kappa_1,\epsilon_1)$ For $\forall \rho\geq \delta$, we have $\nu(B(a,\rho))\leq \delta^{-\epsilon_1}\rho^{\kappa_1}$.
			\item[(2)]$(\delta,\kappa,\epsilon)$ For $\rho\in[\delta,\delta^{\epsilon} ]$, we have $\nu(B(a,\rho))\leq \rho^{\kappa}$.
		\end{itemize}
		Then $(2)(\delta,\kappa,\epsilon)$ implies $(1)(\delta,\min\{\kappa,1\},\epsilon)$ and if $\kappa_1>2\epsilon_1$, we have that  $(1)(\delta,\kappa_1,\epsilon_1)$ implies $(2)(\delta,\kappa_1/2,2\epsilon_1/\kappa_1)$.
\end{lem}
\begin{proof}
	$(2)\Rightarrow(1)$ For $\rho<\delta^{\epsilon}$, it is obvious. For $\rho>\delta^{\epsilon}$, we use the trivial bound
	\[\nu(B(a,\rho))\leq 1\leq \delta^{-\epsilon+\epsilon\min\{\kappa,1\}}. \]
	Hence $(2)$ implies that $(1)$ holds for $(\epsilon_1,\kappa_1)=(\epsilon,\min\{\kappa,1 \})$.
	
	$(1)\Rightarrow(2)$ We want to find $(\epsilon,\kappa)$ such that $(2)$ holds. Let $\rho=\delta^{t}$. That means
	\begin{equation*}
	\epsilon_1-t\kappa_1\leq -t\kappa \text{ for }t\in [\epsilon,1].
	\end{equation*}
	By $\kappa_1>2\epsilon_1$, we can take $(\epsilon,\kappa)=(2\epsilon_1/\kappa_1,\kappa_1/2)$.
\end{proof}
The assumption $(2)$ in Lemma \ref{lem:equivalent} is the original definition of Bourgain. This assumption roughly says that the measure $\nu$ has dimension $\kappa$ at scale $\delta$ to scale $\delta^\epsilon$. The assumption $(1)$ is more convenient to be used. The smaller the parameter $\epsilon_1$ is, the more regularity the measure $\nu$ has.

Our result on the discretized sum-product estimates relies on a result of He and de Saxcé. They study sum-product phenomenon in finite dimensional linear representations of Lie groups. We will state the version we need, their theorem is much more general.
\begin{defi}
	Recall that we consider the action of $(\R^*)^n$ on $V=\R^n$ given by multiplication in each coordinate. Let $W$ be a linear subspace of $V$ such that $W$ is not a submodule, that is, there exists $g$ in $(\R^*)^n$ such that $gW\nsubseteq W$. Then we call $Stab_{(\R^*)^n}(W)$ a \textbf{proper stabilizer}.
\end{defi}
Let $A$ be a subset of $(\R^*)^n$ and let $X$ be a subset of $\R^n$. For $s\geq 1$, we define $\l A,X\r_s$ to be the set of elements which can be obtained as sums, differences and products of at most $s$ elements of $A$ and $X$. For example, we have $\l A,X\r_s=\{\pm g_{1,1}\cdots g_{1,i_1}v_1\pm\cdots \pm g_{l,1}\cdots g_{l,i_l}v_l |\ i_1,\cdots ,i_l,l\in \bb N,\ i_1+\cdots i_l+l\leq s  \}$. 
\begin{prop}\cite[Thm.2.3]{he_sum-product_2018}\label{prop:sumrep}
		Recall that we consider the action of $(\R^*)^n$ on $V=\R^n$ given by multiplication in each coordinate. There exists a neighbourhood $U$ of the identity in $(\R^*)^n$ such that the following holds. Given $\epsilon_0,\kappa>0$, there exist $s\geq 1$ and $\epsilon>0$ such that for all $\delta>0$ sufficiently small, if $A\subset U$ and $X\subset B_V(0,1)$ satisfy the following $(\delta,\kappa,\epsilon)$ assumption:
	\begin{itemize}
		\item[(i)] For  $j=1,\dots, n$
		\begin{equation*}
		\forall\rho\geq\delta,\ \nrho(\pi_j(A))\geq\delta^\epsilon\rho^{-\kappa},
		\end{equation*} 
		where $\pi_j$ denotes the projection into $j$-th coordinate,
		
		\item[(ii)] $A$ is $\delta^\epsilon$ away from proper stabilizers,
		
		\item[(iii)]$X$ is $\delta^\epsilon$ away from coordinate subspaces.
	\end{itemize}
	Then,
	\begin{equation*}
	B_V(0,\delta^{\epsilon_0})\subset\l A,X \r_s+B_V(0,\delta).
	\end{equation*}
\end{prop}
% restrict to $\R^n$
%We restrict our attention to $E=\bb R^n$. We apply Proposition \ref{prop:sumrep} to the case that $V$ is the semisimple algebra $\RR^n$ and $G=E^*$ is the group of invertible elements of $\RR^n$. 
%Recall that $E\simeq\Pi_{1\leq j\leq n}E_j $, where each $E_j$ is a simple algebra.
We will use the ring structure of $\R^n$. Recall that for a subset $A$ of $(\R^*)^n$, which is also a subset of $\R^n$, we define $\l A\r_s$ as $\l A,X\r_s$ with $X=A$. As a corollary of Proposition \ref{prop:sumrep}, we have
\begin{prop}\label{prop:sumproalg}
	There exists a neighbourhood $U$ of the identity in $(\R^*)^n$ such that the following holds. 
	Given $\kappa>0, \epsilon_0>0$, there exist $\epsilon>0$ and $s>0$ such that, for $\delta$ sufficiently small, if $A$ is a subset of $U$ satisfies the following $(\delta,\kappa,\epsilon)$ assumption:
	
	(i)For $j=1,\dots, n$
	\begin{equation*}
	\forall\rho\geq\delta,\ \nrho(\pi_j(A))\geq\delta^\epsilon\rho^{-\kappa},
	\end{equation*} 
	where $\pi_j$ denotes the projection into $j$-th coordinate, 
	
	(ii)$A$ is $\delta^\epsilon$ away from maximal proper unitary subalgebras.
	
	Then we have
	\[B_{\R^n}(0,\delta^{\epsilon_0})\subset \l A\r_s+B_{\R^n}(0,\delta). \]
\end{prop}
\begin{proof}
	Take $X=A-A$. We can shrink $U$ to ensure that $X\subset U-U\subset B_{\R^n}(0,1)$. Then we claim that $A,X$ satisfies $(\delta,\kappa,2\epsilon/\kappa)$ assumption of Proposition \ref{prop:sumrep}.
	
	Assumption (i) of Proposition \ref{prop:sumrep} is the same as Assumption (i) of this proposition. For assumption (iii) of Proposition \ref{prop:sumrep}, take $\rho=\delta^{2\epsilon/\kappa}$. Then $$\nrho(\pi_j(X))\geq\nrho(\pi_j(A))\geq\delta^\epsilon\rho^{-\kappa}= \delta^{-\epsilon}>1.$$ 
	Hence, $X$ is $\delta^{2\epsilon/\kappa}$ away from coordinate subspaces. The assumption(iii) in Proposition \ref{prop:sumrep} is satisfied.
	
	It remains to verify Assumption (ii) of Proposition \ref{prop:sumrep}. We need to change the point of view. The set $G=(\R^*)^n\subset \R^n$ is seen as subsets of $Aut(V)\subset End(V)$, the automorphism group and the endomorphism ring of $V$. The main point is that in the case of $\R^n$, proper stabilizers are contained in the subalgebras. In other words, let $W$ be a subspace of $V$ which is not a $G$-submodule. Then the proper stabilizer satisfies
	\begin{align*}
	Stab_{G}(W)=G\cap Stab_{\R^n}(W)=G\cap \{a\in \R^n| a\cdot W\subset W \}.
	\end{align*}
	By definition, $Stab_{G}W$ is a proper subgroup of $G$. The fact that $Stab_{\R^n}(W)$ is a unitary subalgebra of $\R^n$ implies that $Stab_{\R^n}(W)$ must be a proper unitary subalgebra of $\R^n$.	Hence, the assumption (ii) of Proposition \ref{prop:sumrep} is automatically satisfied.
	
	Applying Proposition \ref{prop:sumrep} with $\kappa,\epsilon_0$ implies that there exists $s_1$ such that
	\[B_{\R^n}(0,\delta^{\epsilon_0})\subset \l A,X\r_{s_1}+B_{\R^n}(0,\delta), \]
	when $\epsilon$ is small enough. The observation that
	\[\l A,X\r_{s_1}=\l A, A-A\r_{s_1}\subset\l A\r_{2s_1} \]
	implies the result.
\end{proof}

%Let $A$ be a bounded subset of $\R^n$. Let $\l A\r_s$ be the set of elements which are obtained by taking sum or multiplication of elements in $A$ at most $s$ times.
%
\begin{lem}\label{lem:11}
	Let $K>2$ be a roughness constant and let $A$ be a subset of $B_{\R^n}(0,K)$. If 
	\[\ndelta(A+A)+\ndelta(A+A\cdot A)\leq K\ndelta(A), \]
	then for any integer $s$
	\[\ndelta(\l A\r_s)\leq K^{O_s(1)}\ndelta(A). \]
\end{lem}
(See \cite[Lemma 11]{he_discretized_2016} and \cite[Lemma 4.5]{breuillard2011lectures} for more details)
This lemma tells us that instead of proving that $A+A$ or $A+A\cdot A$ is large, it is sufficient to prove that $\l A\r_s$ is substantially large.

As a by-product, using Lemma \ref{lem:11}, we have the following version of discretized sum-product estimates in $\R^n$.
\begin{comment}
	content...

%, which generalizes the version of simple algebra in \cite{he_discretized_2016}. A semisimple real algebra of finite dimension is isomorphic to a direct sum of simple algebras, that is
\begin{equation*}
E\simeq\Pi_{1\leq j\leq n}E_j,
\end{equation*}
where each $E_j$ is a simple algebra, called $j$-th simple component. 
%Simple algebras
All the algebras considered here is finite dimensional over $\RR$, and with a norm. The norms on finite dimensional vector space are equivalent, different choices of norm will only change the constants in the result.
\end{comment}
\begin{prop}\label{prop:semisimple}
%	Let $\R^n$ be a normed semisimple real algebra of finite dimension. 
There exists a neighbourhood $U$ of the identity in $(\R^*)^n$ such that the following holds. Given $\kappa>0,\sigma\in(0,n)$, there exists $\epsilon>0$ such that for all $\delta>0$ sufficiently small, if $A\subset U$ satisfies the following:
	
	(i) For $j=1,\dots, n$
	\begin{equation*}
	\forall\rho\geq\delta,\ \nrho(\pi_j(A))\geq\delta^\epsilon\rho^{-\kappa},
	\end{equation*} 
	where $\pi_j$ denotes the projection into $j$-th coordinate, 
	
	(ii) $A$ is $\delta^\epsilon$ away from proper unitary subalgebras of $\R^n$,
	
	(iii) $\ndelta(A)\leq \delta^{-\sigma-\epsilon}$.
	
	Then
	\begin{equation*}
	\ndelta(A+A)+\ndelta(A+A\cdot A)\geq \delta^{-\epsilon}\ndelta(A).
	\end{equation*}
\end{prop}

We deduce Proposition \ref{prop:semisimple} from Lemma \ref{lem:11} and Proposition \ref{prop:sumproalg}. The proof is exactly the same as the proof of \cite[Theorem 2]{he_discretized_2016}. We include its proof for completeness.
\begin{proof}[Proof of Proposition \ref{prop:semisimple}]
	Suppose that the result fails. For every $\epsilon>0$ there exists $A$ satisfying the hypothesis of Proposition \ref{prop:semisimple} but 
	\[\ndelta(A+A)+\ndelta(A+A\cdot A)< \delta^{-\epsilon}\ndelta(A). \]
	We will reach a contradiction when $\epsilon$ is small enough depending only on $\kappa,\sigma$ and $\R^n$.
	
	Then by Lemma \ref{lem:11} and assumption (ii) of Proposition \ref{prop:semisimple}, for every integer $s$, we have
	\begin{equation}\label{equ:ndesl}
	\ndelta(\l A\r_s)\leq \delta^{-O_s(\epsilon)}\ndelta(A)\leq \delta^{-O_s(\epsilon)-\sigma}.
	\end{equation}
	On the other hand, $A$ also satisfies the assumptions of Proposition \ref{prop:sumproalg}. Given $\epsilon_0>0$, there exist $\epsilon_1>0$ and integer $s$ such that if $\epsilon\leq \epsilon_1$, then
	\begin{equation*}
	B_{\R^n}(0,\delta^{\epsilon_0})\subset \l A\r_s+B_{\R^n}(0,\delta).
	\end{equation*}
	Therefore
	\begin{equation}\label{equ:ndesg}
	\ndelta(\l A\r_s)\geq\ndelta(B_{\R^n}(0,\delta^{\epsilon_0}))=\delta^{n(-1+\epsilon_0)}.
	\end{equation}
	If we take $\epsilon_0$ sufficiently small such that $n(1-\epsilon_0)> \sigma$, and take $\epsilon$ sufficiently small such that
	\[n(1-\epsilon_0)>O_s(\epsilon)+\sigma, \]
	then \eqref{equ:ndesl} contradicts \eqref{equ:ndesg}.
\end{proof}
This version is not sufficient to imply the decrease of Fourier transform of multiplicative convolution of measures. We will introduce more tools of additive combinatorics to obtain a stronger form of discretized sum-product estimates.
\subsection{Basics of discretized sets}
Before proving our results, we recall some elementary and known results in the discretized setting. Let $\delta>0$ be the scale. Let $K\geq 2$ be a roughness constant. Two quantities bounded by a polynomial of $K$ is considered as equivalent.
\begin{lem}\label{lem:klip}
	Let $f$ be a $K$-Lipschitz function from $\RR^n$ to $\RR^n$. Let $A$ be a bounded subset of $\RR^n$. We have
	\begin{equation}\label{equ:klip}
		\ndelta(fA)\ll K^n\ndelta(A).
	\end{equation}
\end{lem}
\begin{defi}
	For a bounded subset $A$ of $\RR^n$, we denote by $A^{(\delta)}$ the $\delta$-neighbourhood of $A$, given by 
	\[A^{(\delta)}=\{x\in\RR^n|d(x,A)\leq\delta \}. \]
\end{defi}
\begin{lem}Let $A$ be a bounded subset of $\RR^n$. Let $\tilde{A}$ be a maximal $\delta$-separated subset of $A$, that is, different elements of $\tilde{A}$ have distance at least $\delta$ and $\tilde{A}$ is maximal for inclusion. Then
	\begin{equation}\label{equ:neiequ}
		\ndelta(A)\sim |A^{(\delta)}|\delta^{-n}\sim \#\tilde{A},
	\end{equation}
	where $|A|$ denotes the volume of $A$ and $\#\tilde{A}$ denotes the number of elements of $\tilde{A}$.
\end{lem}
\begin{defi}[Ruzsa distance]
	Let $A, B$ be two bounded subsets of $\RR^n$. We define the Ruzsa distance of $A,B$ at scale $\delta$ by
	\[d_\delta(A,B)=\frac{1}{2}\log\frac{\ndelta(A-B)^2}{\ndelta(A)\ndelta(B)}.\]
	This is not a real distance. It measures the additive structure of $A$ and $B$. 
\end{defi}
\begin{lem}[Ruzsa triangular inequality]\label{lem:ruztri}
	Let $A,B,C$ be three bounded subsets of $\RR^n$. Then
	\begin{equation}\label{equ:ruztri}
		\ndelta(B)\ndelta(A-C)\ll\ndelta(A-B)\ndelta(B-C).
	\end{equation}
\end{lem}
The above inequality \eqref{equ:ruztri} is roughly a triangular inequality for the Ruzsa distance $d_\delta$.
\begin{lem}[Pl\"unnecke-Ruzsa inequality]\label{lem:pluruz}
	Let $A,B$ be two bounded subsets of $\RR^n$. If $\ndelta(A+B)\leq K\ndelta(B)$, then for $k,l$ in $\bb N$ we have
	\[\ndelta(kA-lA)\leq O(K)^{k+l}\ndelta(B). \]
\end{lem}
In \cite{he_orthogonal_2017}, He explains how to deduce the discretized version from the discrete version of the above two lemmas. For the discrete version, please see \cite{tao_additive_2006}. The main ingredient of proof is the Ruzsa covering lemma.
\begin{defi}
	Let $A, B$ be two bounded subsets of $\RR^n$. We define the doubling constant of $A$ at scale $\delta$ by 
	\[\sigma_\delta[A]:=\frac{\ndelta(A+A)}{\ndelta(A)}=\exp(d_\delta(A,-A)). \]
	We write $A\approx_KB$ if $\ndelta(A+B)\leq K\ndelta(A)^{1/2}\ndelta(B)^{1/2}$, which is equivalent to that the Ruzsa distance is small, that is to say $d_\delta(A,-B)\leq \log K$.
\end{defi}
\begin{lem}[Ruzsa calculus]\label{lem:ruzsa}
	Let $A,B,C$ be three bounded subsets of $\RR^n$.  Then
	\begin{itemize}
		\item[(1)] If $A\approx_K B$, then $A\approx_{K^{O(1)}}-B$, $K^{-O(1)}\ndelta(B)\leq \ndelta(A)\leq {K^{O(1)}}\ndelta(B)$ and $\sigma_\delta[A], \sigma_\delta[B]\leq K^{O(1)}$.
		\item[(2)] If $A\approx_KB$ and $B\approx_KC$, then $A\approx_{K^{O(1)}}C$.
		\item[(3)] If $\sigma_\delta[A], \sigma_\delta[B]\leq K$ and $\ndelta(A^{(\delta)}\cap B^{(\delta)})\geq K^{-1}\ndelta(A)^{1/2}\ndelta(B)^{1/2}$, then $A\approx_{K^{O(1)}}B$.
	\end{itemize}
\end{lem}
The proofs are direct applications of the Ruzsa triangular inequality and the Pl\"unnecke-Ruzsa inequality. For the discrete version, please see \cite{tao_additive_2006} and the second note of Green in \cite{green2009additive}. The first and second statements mean that the Ruzsa distance is symmetric and transitive. The Ruzsa calculus will be used to prove Proposition \ref{prop:BSG} (Additive-Multiplicative Balog-Szemerédi-Gowers theorem).
\begin{comment}
For completeness, we give a sketch of the proof
	(1) \[\ndelta(A-A)\leq \ndelta(A+B)\ndelta(-B-A)/\ndelta(B)\ll\ndelta(A) \]
		\[\ndelta(A+A)\leq \ndelta(B) \]
	\[\ndelta(A-B)\leq \ndelta(A+A)\ndelta(-A-B)/\ndelta(B)\ll \ndelta(A) \]

	(2)
	\[\ndelta(A+C)\leq \ndelta(A-B)\ndelta (B+C)/\ndelta(B) \]

	(3) First prove the inequality with $A^{(\delta)},B^{(\delta)}$. Then the result of $A,B$ follows directly.
	\[\ndelta (A+B)\leq \ndelta(A-A\cap B)\ndelta(A\cap B+B)/\ndelta (A\cap B) \]
\end{comment}

%
%
\subsubsection*{The additive energy: the discrete case}
We first introduce the additive energy in the discrete case. Let $A,B$ be two finite sets in an abelian group $G$. We define the additive energy $\omega(+,A\times B)$ as the number of the quadruplet $(a,b,a',b')$ in $A\times B\times A\times B$ such that $a+b=a'+b'$, that is,
\[\omega(+,A\times B)=\#\{(a,b,a',b')\in A\times B\times A\times B| a+b=a'+b' \}. \]
We also have a formulation with $\ell^2$ norm
\begin{equation}\label{equ:omeab}
\omega(+,A\times B)=\|\BB_A*\BB_B\|_2^2 ,
\end{equation}
where the measure in defining $\ell^2$ norm is the counting measure. From the definition, by Young's inequality, we have 
\begin{equation}\label{equ:optimal value}
\omega(+,A\times B)\leq |A|^{3/2}|B|^{3/2},
\end{equation}
where $|A|$ denotes the number of elements in $A$. The additive energy is important because it reflects the additive structure of $A$ and $B$. If $|A+B|\leq K|A|^{1/2}|B|^{1/2}$, then by the Cauchy-Schwarz inequality,
\begin{equation}
\omega(+,A\times B)\geq\frac{|A|^2|B|^2}{|A+B|}\geq K^{-1}|A|^{3/2}|B|^{3/2}, 
\end{equation}
which is robustly large with respect to the optimal value of $\omega(+,A\times B)$ \eqref{equ:optimal value}.
  (See \cite{tao_additive_2006} and \cite{green_sum-product_2009} for more details).
 \subsubsection*{The additive energy: the continuous case}
 We now define the discretized version of the additive energy. On a Cartesian product $X\times Y$ of metric spaces, we use the distance defined by
 \[d((x,y),(x',y'))=\sqrt{d_X^2(x,x')+d_Y^2(y,y')}, \]
 where $x,x'$ are in $X$ and $y,y'$ are in $Y$. 
\begin{defi}[Energy of a map]
	Let $X,Y$ be two metric spaces, and let $\varphi$ be a Lipschitz map from $X$ to $Y$. For a subset $C$ of $X$, the energy of $\varphi$ at scale $\delta$ is defined by
	\begin{equation}
		\omega_\delta(\varphi,C)=\ndelta(\{(a,a')\in C\times C|d(\varphi(a),\varphi(a'))\leq \delta \}).
	\end{equation}
\end{defi}
\begin{lem}\label{lem:12}
	Let $\varphi$ be a $K$-Lipschitz map from $\bb R^m$ to $\bb R^n$, and let $C$ be a bounded subset of $\bb R^m$. Then 
	\begin{itemize}
		\item[(i)] We have
		\begin{equation}
			\ndelta(C)^2\gg \omega_\delta(\varphi,C)\gg_{n,m} \frac{\ndelta(C)^2}{\ndelta(\varphi(C))}.
		\end{equation}
		\item [(ii)] Let $\tilde C$ be a maximal $\delta-$separated subset of $C$. Then
		\begin{equation}
			\omega_\delta(\varphi,C)\ll \#\{(a,a')\in \tilde{C}^2|d(\varphi(a),\varphi(a'))\leq (1+2K)\delta \}.
		\end{equation}
	\end{itemize}
\end{lem}
(See \cite[Lemma 12]{he_discretized_2016} for more details) When $m=2n$, $C=A\times B\subset\RR^{2n}$ with $A,B$ in $\RR^n$ and $\varphi(a,b)=a+b$, we call $\omega_\delta(+,A\times B)$ the additive energy of $A,B$ at scale $\delta$. We have a formulation with $L^2$ norm (see \cite[Appendix A.1]{boutonnet2017local} for example. This is also the discretized version of \eqref{equ:omeab}.) We have an inequality
\begin{equation}\label{equ:additive energy}
\omega_\delta(+,A\times B)\gg \delta^{-3n}\|\BB_{A}*\BB_{B}\|_2^2.
\end{equation}

Lemma \ref{lem:12} (i) implies that 
\begin{equation}\label{equ:omegg1}
\omega_\delta(+,A\times B)\gg \frac{\ndelta(A\times B)^2}{\ndelta(A+B)}\geq \frac{\ndelta(A)^2\ndelta(B)^2}{\ndelta(A+B)}.
\end{equation}
If $A\approx_KB$, that is, $\ndelta(A+B)\leq K\ndelta(A)^{1/2}\ndelta(B)^{1/2}$, then \eqref{equ:omegg1} implies
\begin{equation}\label{equ:omegg}
	\omega_\delta(+,A\times B)\gg K^{-1}\ndelta(A)^{3/2}\ndelta(B)^{3/2}.
\end{equation}
This means that when two sets $A, B$ have additive structure then the additive energy is relatively large.

The additive energy is powerful when combined with the following proposition, a partial converse to \eqref{equ:omegg}, which says that if two sets have relatively large additive energy, then there exist large subsets which have additive structure.
\begin{prop}[Balog-Szemer\'edi-Gowers]\cite[Theorem 6.10]{tao_product_2006}\label{prop:bsg}
	Let $A,B$ be two bounded subsets of $\bb R^n$ such that 
	\[\omega_\delta(+,A\times B)\geq K^{-1}\ndelta(A)^{3/2}\ndelta(B)^{3/2}. \]
	Then there exist subsets $A', B'$ of $A,B$ such that 
	$$\ndelta(A')\gg_n K^{-O(1)}\ndelta(A),\ \ndelta(B')\gg_n K^{-O(1)}\ndelta(B)$$
	 and 
	\[\ndelta(A'+B')\ll_n K^{O(1)}\ndelta(A')^{1/2}\ndelta(B')^{1/2}. \]
\end{prop}
\subsection{Sum-product estimates in $\RR^n$}
%
\begin{comment}
\end{comment}
%
We first state the discrete version of the growth under a ball to give a flavour of this type of phenomenon and to motivate the continuous version.
\begin{lem}\cite[Lemma 3.1]{green_sum-product_2009}
	Let $p$ be a prime number. If $X$ is a subset of $\bb F_p$, then 
	\[\sup_{a\in\bb F_p}|X+aX|\geq \frac{1}{2}\min\{|X|^2,p \}. \] 
\end{lem}
The proof is by calculating the additive energy in two ways. Suppose that the result does not hold, then the additive energy $\omega(+,X\times aX)$ is large for every $a$ in $\bb F_p$. But the sum of the additive energy $\omega(+,X\times aX)$ with respect to $a$ in $\bb F_p$ is small, which gives the contradiction.

The continuous version uses a Fubini type argument to study the growth under a ball in $(\bb R^*)^n$. Recall that $\id=(1,\cdots,1)$ is the identity in $(\R^*)^n$.

%The main difficult of semisimple algebra is that we need to mimic the prove of $End(\bb R^n)$ on $\bb R^n$. Get a version of $End(\bb R^n)\times End(\bb R^m)$ on $\bb R^n\oplus\bb R^m$
\begin{lem}\label{lem:plusreg}
	Given $\kappa>0,\sigma\in(0,n)$, there exists $\epsilon>0$ such that for $\delta$ sufficiently small the following holds. Let $X$ be a bounded subset of $\RR^n$  such that for $j=1,\dots n$
	\begin{equation*}
	\forall\rho>\delta,\ \nrho(\pi_j(X))\geq\delta^\epsilon\rho^{-\kappa}
	\end{equation*}
	and $\ndelta(X)\leq \delta^{-\sigma-\epsilon}$. Then
	\begin{equation*}
	\sup_{a\in B_{\RR^n}(\id,1/2)}\ndelta(X+aX)\geq \delta^{-\epsilon}\ndelta(X).
	\end{equation*}
\end{lem}
\begin{rem}
	We follow closely the proof of \cite[Theorem 3]{he_discretized_2016}.
	To prove the stronger version, we need another lemma, which is a reducible version of \cite[Prop.29]{he_discretized_2016}. The proof is essentially the same as the irreducible version, with the estimate of small balls replaced by thin cylinders. %It is reasonable since in the semisimple algebra case we need more regularity on the projection into ideals. 
\end{rem}
\begin{proof}[Proof of Lemma \ref{lem:plusreg}]
	Assume that $\ndelta(X+X)<\delta^{-\epsilon}\ndelta(X)$, otherwise the proof is finished. For $\rho>\delta$ and $j=1,\cdots, n$, we have
	\begin{equation}\label{equ:x+x}
	\ndelta(X+X)\geq\nrho(\pi_j(X))\max_{b\in\bb R}\ndelta(X\cap\pi_j^{-1}B_{\RR}(b,\rho)).
	\end{equation}
	This can be proved by the following standard argument. Choose a maximal subset $\{ c_i\}$ of $X$ such that $\pi_j(c_i)$ is $2\rho$-separated. Fix $b$ in $\RR$. Choose a maximal $\delta$-separated subset $\{d_k\}$ of $X\cap\pi_j^{-1}B_{\RR}(b,\rho)$. If $(i,k)\neq (i',k')$, then
	\[d(c_i+d_k, c_{i'}+d_{k'})\geq\delta. \]
	 Hence $\{c_i+d_k\}_{i,k}$ is a $\delta$-separated subset of $X+X$ and \eqref{equ:x+x} follows.
	
	For all $b$ in $\RR$, by \eqref{equ:x+x} and hypothesis
	\begin{equation}\label{equ:ndpi}
	\ndelta(X\cap\pi_j^{-1}B_{\RR}(b,\rho))\leq\frac{\ndelta(X+X)}{\nrho(\pi_j(X))}\leq \frac{\delta^{-\epsilon}\ndelta(X)}{\delta^\epsilon\rho^{-\kappa}}= \delta^{-2\epsilon}\rho^\kappa\ndelta(X).
	\end{equation}
	
	Let $\mu$ be the normalized Lebesgue measure on $B_{\RR^n}(\id,1/2)$ with total mass $1$, and let $a$ be a random variable following the law of $\mu$. Define $\varphi_a:\bb R^n\times\bb R^n\rightarrow\bb R^n$ by
	\[\varphi_a(x,y)=x+ay. \]
	By Lemma \ref{lem:12} (i),
	\[\ndelta(\varphi_a(X\times X))\gg\frac{\ndelta(X\times X)^2}{\omega_\delta(\varphi_a,X\times X)}, \]
	which is also
	\[\ndelta(X+aX)\gg\frac{\ndelta(X)^4}{\omega_\delta(\varphi_a,X\times X)}. \]
	By the Jensen inequality on the function $t\mapsto \frac{1}{t}$ from $\bb R^+$ to $\bb R^+$,
	\begin{equation}\label{equ:jensen}
	\bb E(\ndelta(X+aX))\gg\frac{\ndelta(X)^4}{\bb E(\omega_\delta(\varphi_a,X\times X))}.
	\end{equation}
	Therefore it is sufficient to give a bound that $\bb E(\omega_\delta(\varphi_a,X\times X))\ll\delta^\epsilon\ndelta(X)^3$.
	
	By Lemma \ref{lem:12} (ii), letting $\tilde{X}$ be a maximal $\delta$-separated subset of $X$, we have
	\begin{equation}\label{equ:energy}
	\begin{split}
	\bb E(\omega_\delta(\varphi_a,X\times X))&\ll\bb E(\#\{(x,x',y,y')\in\tilde{X}^4|\|(x-x')+a(y-y') \|\leq 5\delta \})\\
	&=\sum_{x,x',y,y'\in \tilde{X}}\bb P\{\|a(y-y')+(x-x')\|\leq 5\delta \},
	\end{split}
	\end{equation}
	where $a$ is contained in $B_{\RR^n}(\id,1/2)$ and $K=2$.
	Let $\rho$ be a parameter to be fixed later. We distinguish two cases
	\begin{itemize}
		\item If $\min_j|y^j-(y')^j|\geq\rho$, then
		\begin{equation}\label{equ:mingeq}
		\bb P\{\|a(y-y')+(x-x')\|\leq 5\delta \}\ll\delta^n\rho^{-n}.
		\end{equation}
		\item Otherwise, the number of pairs $(y,y')$ such that $\min_j|y^j-(y')^j|<\rho$ can be bounded using \eqref{equ:ndpi} and \eqref{equ:neiequ}
		\begin{equation}\label{equ:minleq}
		\#\{(y,y')\in \tilde{X}^2|\min_j|y^j-(y')^j|<\rho\}\leq \#\tilde{X}(\sum_j\max_{b\in\bb R}\#\{\tilde{X}\cap\pi_j^{-1}B_{\RR}(b,\rho) \})\ll\delta^{-2\epsilon}\rho^\kappa\ndelta(X)^2. 
		\end{equation}
		Moreover, we have for all $x,y,y'\in\tilde{X}$,
		\begin{equation}\label{equ:xyy'}
		\sum_{x'\in\tilde{X}}\bb P\{\|a(y-y')+(x-x')\|\leq 5\delta \}\ll 1 ,
		\end{equation}
		since for every event, there exists a finite number of $x'$ which satisfies the assumption.
	\end{itemize}
	Therefore combining \eqref{equ:energy} \eqref{equ:mingeq} \eqref{equ:minleq} and \eqref{equ:xyy'}, and taking $\rho=\delta^{\frac{n-\sigma}{n+\kappa}}$,
	\begin{align*}
	\bb E(\omega_\delta(\varphi_a,X\times X))&\ll \ndelta(X)^4\delta^n\rho^{-n}+\ndelta(X)^3\delta^{-2\epsilon}\rho^\kappa\\
	&\ll \ndelta(X)^3(\delta^{n-\sigma-\epsilon}\rho^{-n}+\delta^{-2\epsilon}\rho^\kappa)\\
	&\ll\ndelta(X)^3\delta^{-2\epsilon+\frac{\kappa(n-\sigma)}{n+\kappa}}.
	\end{align*}
	When $\epsilon$ is sufficiently small, we have $\bb E(\omega_\delta(\varphi_a,X\times X))\ll \ndelta(X)^3\delta^{\epsilon}$, which finishes the proof.
\end{proof}
Before proving Theorem \ref{prop:sumpro}, we need to introduce $S_\delta$, the set of ``good elements".
Let $A$ be a bounded subset of $\bb R^n$. Let 
\[S_\delta(A,K)=\{a\in B_{End(\bb R^n)}(0,K)| \ndelta(A+aA)\leq K\ndelta(A) \}. \]
The following lemma says that $S_\delta(A,K)$ has a ``ring structure".
\begin{lem}\label{lem:30}
	Let $A\subset B(0,K)$ be a subset of $\bb R^n$.
	\begin{itemize}
		\item[(i)] If $a$ is in $S_\delta(A,K)$ and $\|a-b\|\leq K\delta$, then $b$ belongs to $S_\delta(A,K^{O(1)})$.
		\item[(ii)] If $\id, a, b$ are in $S_\delta(A,K)$, then $a-b, a+b, ab$ belong to $S_\delta(A,K^{O(1)})$.
		\item[(iii)] Suppose that $a$ is invertible. If $a^{-1}$ is in $B(0,K)$ and $a$ is in $S_\delta(A,K)$, then $a^{-1}$ belongs to $S_\delta(A,K^{O(1)})$.
	\end{itemize}
\end{lem}
(See \cite[Lemma 30]{he_discretized_2016} and \cite[Proposition 3.3]{bourgain2004sum} for more details)
\begin{proof}[Proof of Theorem \ref{prop:sumpro}]
	The idea is to use Proposition \ref{prop:sumproalg} to force $A$ to grow to a fat ball. Then Lemma \ref{lem:plusreg} implies the growth of regularity under the action of a fat ball.
	
	Assume that the result fails. that is, for every $\epsilon>0$, there exist $A,X$ satisfying the assumptions of Theorem \ref{prop:sumpro} such that \begin{equation}\label{equ:asubs}
		A\subset S_\delta(X,\delta^{-\epsilon}).
	\end{equation}
	We will reach a contradiction when $\epsilon$ is small enough depending on $\kappa,\sigma$.
	
	By Proposition \ref{prop:sumproalg}, for every $\epsilon_0>0$, there exist $s\in\bb N$ and $\epsilon_1>0$ depending only on $\epsilon_0$ and $\kappa$, such that if $\epsilon<\epsilon_1$ then
	\begin{equation}\label{equ:bdels}
	B_{\RR^n}(0,\delta^{\epsilon_0})\subset\l A\r_s+B_{\RR^n}(0,\delta).
	\end{equation}
	By Lemma \ref{lem:30} (ii) with $K=\delta^{-\epsilon}$ and \eqref{equ:asubs}, we have 
	\begin{equation}\label{equ:ass}
		\l A\r_s\subset S_\delta(X,\delta^{-O(s)\epsilon}).
	\end{equation}
	By Lemma \ref{lem:30} (i) with $K=\delta^{-O(s)\epsilon}$ and \eqref{equ:ass}, \eqref{equ:bdels}
	\begin{equation}\label{equ:Bdele}
	 B_{\RR^n}(0,\delta^{\epsilon_0})\subset S_\delta(X,\delta^{-O(s)\epsilon}).
	\end{equation} 
	By Lemma \ref{lem:30} (iii) with $K=\delta^{-O(s)\epsilon-\epsilon_0}$, $a=\frac{\delta^{\epsilon_0}}{2}\id$ and \eqref{equ:Bdele}, we have
	\begin{equation}\label{equ:2del}
	2\delta^{-\epsilon_0}\id=a^{-1}\in S_\delta(X,\delta^{-O(s)\epsilon-O(\epsilon_0)}).
	\end{equation} 
	Again by Lemma \ref{lem:30} (ii), using product and \eqref{equ:Bdele}, \eqref{equ:2del}, we obtain
	\begin{equation}\label{equ:bid1}
		B_{\RR^n}(\id,1/2)\subset B_{\RR^n}(0,2)=2\delta^{-\epsilon_0}\id\cdot B_{\RR^n}(0,\delta^{\epsilon_0})\subset S_\delta(X,\delta^{-O(s)\epsilon-O(\epsilon_0)}). 
	\end{equation}
	By Lemma \ref{lem:plusreg}, there exists $\epsilon_2>0$ depending only on $\sigma$ and $\kappa$, such that when $\epsilon<\epsilon_2$
	\begin{equation}\label{equ:ainb1}
	\sup_{a\in B_{\RR^n}(\id,1/2)}\ndelta(X+aX)\geq \delta^{-\epsilon_2}\ndelta(X).
	\end{equation}
	
	Taking $\epsilon_0$ sufficiently small, and then taking $\epsilon$ sufficiently small such that $O(s)\epsilon+O(\epsilon_0)<\epsilon_2$, we get a contradiction from \eqref{equ:bid1} \eqref{equ:ainb1}
	\[\delta^{-O(s)\epsilon-O(\epsilon_0)}\ndelta(X)\geq \sup_{a\in B_{\RR^n}(\id,1/2)}\ndelta(X+aX)\geq \delta^{-\epsilon_2}\ndelta(X). \]
	The proof is complete.
\end{proof}
\section{Application to multiplicative convolution of measures}
\label{sec:appmul}
%We use the method of Bourgain in \cite{bourgain2010discretized}, the main difference is in using the sum-product estimates. The idea is to prove a type of $L^2$-flattening lemma using proof by contradiction. Hence we use a uniform version of Balog-Szemerédi-Gowers theorem, which is due to Bourgain.

\textbf{Notation}: For a measure $\mu$ on $\bb R^n$, let $\mu^-$ be the symmetry of $\mu$, that is, $\mu^-(E)=\mu(-E)$ for any Borel set $E$ of $\bb R^n$. Let $\mu^{(r)}$ be the $r$-times additive convolution of $\mu$. Recall that $\mu_k$ is the $k$-times multiplicative convolution of $\mu$. For an element $x$ in $\RR^n$, we write $x^j$ for its $j$-th coordinate, $j=1,2,\cdots,n$. Recall that we use the norm induced by the standard scalar product on $\bb R^n$, that is to say for $x\in\RR^n$, $\|x\|=\sqrt{(x^1)^2+\cdots +(x^n)^2}$. All vectors $x,\xi$ in $\bb R^n$ are column vectors. For $y$ in $\R^n$ and measure $\nu$ on $\R^n$, let $(m_y)_*\nu$ be the pushforward measure of $\mu$ by the multiplication action of $y$, that is, $(m_y)_*\nu(E)=\nu(y^{-1}E)$. In order to simplify the notation, we abbreviate $B_{\RR^n}(0,R)$ to $B(0,R)$. For a function $f$ on $\RR^n$, we write $\|f\|_p$, $p=1,2,\infty$, for its $L^p$ norm on $\RR^n$.%Let $|x|=\min\{|x^1|,|x^2| \}$

	Let $P_\delta=\frac{\BB_{B(0,\delta)}}{|B(0,\delta)|}$, where $|\cdot|$ is the Lebesgue measure of a Borel set in $\RR^n$. Let $\nu_\delta=\nu*P_\delta$, which is an approximation of $\nu$ at scale $\delta$. 
\subsection{$L^2$-flattening}
\begin{lem}[$L^2$-flattening]\label{lem:l2fla}
	Given $\sigma_1$, $\kappa>0$, there exists $\epsilon=\epsilon(\sigma_1,\kappa)>0$ such that the following holds for $\delta$ small enough. Let $\nu$ be a symmetric Borel probability measure on $[-\delta^{-\epsilon},\delta^{-\epsilon}]^n\subset\R^n$. Assume that 
	\[ \|\nu_\delta\|_2^2\geq \delta^{-\sigma_1} \]
	and $\nu$ satisfies $(\delta,\kappa,\epsilon)$ projective non concentration assumption, that is, 
	\begin{align}\label{equ:procon}%\label{equ:awasub}
		\forall\rho\geq\delta,\quad\sup_{a\in\bb R,v\in\bb S^{n-1}}(\pi_v)_*\nu(B_\RR(a,\rho))=\sup_{a,v}\nu\{x| \l v, x\r\in B_\RR(a,\rho) \}\leq\delta^{-\epsilon} \rho^{\kappa}.
	\end{align}
	Then 
	\begin{equation}\label{equ:intnu}
		\int	\|\nude*(m_y)_*\nude\|_2^2\dd\nu(y)\leq\delta^{\epsilon} \|\nu_\delta\|_2^2.
	\end{equation}

\end{lem}
\begin{rem}
	The first assumption that the $L^2$ norm is not small means that the measure is not too smooth. Indeed, if the measure is already smooth, then the convolution can not make the measure more smooth. This assumption should be compared with the assumption (iv) in Theorem \ref{prop:sumpro}, where we need that the covering number of the set is not too large. 
	
	By definition and \eqref{equ:procon}, $\|\nude\|_2^2\leq \|\nude\|_\infty\|\nude\|_1\leq \delta^{\kappa-\epsilon-n}$. Hence $\kappa+\sigma_1\leq \epsilon+n$, that is, the non concentration assumption gives a upper bound of $L^2$ norm. Another explication of the $L^2$ norm is in Lemma \ref{lem:nudel}.
\end{rem}

\begin{rem}\label{rem:ideal}
	The non concentration assumption here is stronger than the non concentration in Theorem \ref{prop:sumpro}. This is because we need to make multiplication in the proof. 
	The projective non concentration assumption is preserved under multiplication and addition, with possibly different constants. But the non concentration assumption in Theorem \ref{prop:sumpro} is not.

	The hypothesis of projective non concentration can be weakened to (i) non concentration on coordinate subspaces and (ii) away from linear subspaces. Please see Remark \ref{rem:weak}. But the assumption needed in Theorem \ref{thm:sumfourier} is projective non concentration. Hence we write the same assumption here for simplicity. The step where we really need a projective non concentration is explained in Remark \ref{rem:pro}.
\end{rem}
\begin{rem}\label{rem:nu2}
When $n$ equals 1, this is due to Bourgain \cite{bourgain_erdos-volkmann_2003} \cite{bourgain2010discretized}. It roughly says that under multiplicative and additive convolution the H\"older regularity of a measure will increase, that is, given $\kappa>0$ there exists $\epsilon>0$ such that if for all $x$ in $\RR$ and $r>0$, we have $\nu(B(x,r))\leq r^\kappa$, then $\nu *\nu_2(B(x,r))\leq r^{\kappa+\epsilon}$. With this observation, Bourgain gave a quantitative proof of the Erd\"os-Volkmann ring conjecture \cite[Section 4]{bourgain_erdos-volkmann_2003}. 
\end{rem}
Instead of using the original approach in \cite{bourgain_erdos-volkmann_2003} \cite{bourgain2010discretized}, we will follow the approach used for proving $L^2$-flattening in the case of simple Lie groups, using dyadic decomposition to simplify the argument, developed by Bourgain and Gamburd (see \cite{bourgain2008spectral}, \cite{benoist2016spectral}, \cite{boutonnet2017local} for example).
We introduce an approximation by dyadic level sets.
		\begin{defi}
			Let $\{D_i\}_{i\in I}$ be a family of subsets of $\RR^n$. We call $\{D_i\}_{i\in I}$ an essentially disjoint union, if each point $x$ in $\RR^n$ is covered by at most $C$ different $D_i$, where $C$ is a fixed constant only depending on $\RR^n$.
		\end{defi}
		\begin{lem}\cite{lindenstrauss_hausdorff_2015}\cite[Lemma A.4]{boutonnet2017local}\label{lem:2adic}
			Let $\nu$ be a Borel probability measure on $\RR^n$. Let $\cal C$ be a maximal $\delta$-separated set of $\R^n$. Let $\cal C_0=\{x\in\cal C|0< \nu_{2\delta}(x)\leq 1 \}$ and $\cal C_i=\{x\in\cal C|2^{i-1}<\nu_{2\delta}(x)\leq 2^i  \}$ for $i\geq 1$. For $i\geq 0$, let $X_i=\cup_{x\in\cal C_i}B(x,\delta)$. Then $X_i$ is empty if $i\geq O(\log\frac{1}{\delta})$, and we have
			\begin{itemize}
				\item[(1)] 
				$\nude\ll \sum_{i\geq 0}2^i\BB_{X_i}$ and $\sum_{i> 0}2^i\BB_{X_i}\ll \nu_{3\delta}.$
				\item[(2)] $X_i$ is an essentially disjoint union of balls of radius $\delta$, for each $i\geq 0$.
			\end{itemize}
		\end{lem}
		\begin{lem}\cite[Lemma A.5]{boutonnet2017local}\label{lem:nu4}
			Let $a>0$ and $\nu$ be a Borel probability measure on $\R^n$. Then
			\[\|\nu_{a\delta}\|_2\ll_a \|\nude\|_2. \]
		\end{lem}
		We also need the following inequality, which is an inverse Chebyshev's inequality. Its proof is elementary.
		\begin{lem}\label{lem: inverse cheb}
			Let $K>0$. Let $\nu$ be a probability measure on a measure space $X$. Let $f$ be a nonnegative function on $X$. If $|f(x)|\leq K\int_X f\dd\nu$ on the support of $\nu$, then 
			\[\nu\left\{x\in X\big|f(x)\geq \frac{1}{2}\int_Xf\dd\nu\right \}\geq \frac{1}{2K}. \]
		\end{lem}
	Here is the main idea of the proof of $L^2$-flattening: Suppose that \eqref{equ:intnu} fails. By \eqref{equ:additive energy}, we can obtain two sets with large additive energy from the convolution of its character function. Hence we can find some sets in the support of $\nude$ with large additive energy. Together with Balog-Szemer\'edi-Gowers theorem (Proposition \ref{prop:bsg}), this produces two sets which violate sum-product estimates (Theorem \ref{prop:sumpro}).
	\begin{proof}[Proof of $L^2$-flattening (Lemma \ref{lem:l2fla})] 
	We follow closely the proof of \cite[Lemma 2.5]{benoist2016spectral}. Proof by contradiction: Assume that the result fails. Then for every $\epsilon>0$, there exist $\delta$ small and a measure $\nu$ satisfying
	\begin{equation}\label{equ:nudey1}
	\int	\|\nude*(m_y)_*\nude\|_2^2\dd\nu(y)>\delta^{\epsilon} \|\nu_\delta\|_2^2.
	\end{equation}
	We will reach a contradiction for $\epsilon$ sufficiently small.
	
	For $y$ in $\R^n$ invertible, let $\det y$ be the determinant of $y$ seen as an endomorphism of $\R^n$, that is to say $\det y=y^1\cdots y^n$. Let 
	\begin{equation}\label{equ:ygg}
		E:=\{y\in\R^n|\ |\det y|> \delta^{3n\epsilon/\kappa}=\delta^{O(\epsilon)} \} .
	\end{equation}
	Here $3n\epsilon/\kappa$ depends on the extra parameter $\kappa$, but $\kappa$ is fixed in the proof, so for simplicity we can use $O(\epsilon)$ instead of $O_\kappa(\epsilon)$.
	The complement of $E$ has small $\nu$ measure. Indeed, if $|\det y|\leq \delta^{3n\epsilon/\kappa}$, then there exists $1\leq j\leq n$ such that $|y^j|\leq \delta^{3\epsilon/\kappa}$. From the projective non concentration with $\rho=\delta^{3\epsilon/\kappa}$, we obtain
	\[\nu\{y\in \R^n||y^j|\leq \delta^{3\epsilon/\kappa} \}\leq \delta^{-\epsilon}(\delta^{3\epsilon/\kappa})^\kappa=\delta^{2\epsilon}. \]
	Hence $ \nu(E^c)\leq n\delta^{2\epsilon}$. Then by Young's inequality,
	\[\int_{E^c}\|\nude*(m_y)_*\nude\|_2^2\dd\nu(y)\leq \int_{E^c} \|\nude\|_2^2\|(m_y)_*\nude \|_1^2\dd\nu=\nu(E^c)\|\nude\|_2^2\leq n\delta^{2\epsilon}\|\nude\|_2^2.  \]
	If $\delta$ is small enough depending on $\epsilon$ and $n$, with \eqref{equ:nudey1} this implies 
	\begin{equation}\label{equ:nudey}
		\int_E	\|\nude*(m_y)_*\nude\|_2^2\dd\nu(y)\geq\delta^{\epsilon}(1-n\delta^\epsilon) \|\nu_\delta\|_2^2\geq \frac{\delta^\epsilon}{2}\|\nude\|_2^2.
	\end{equation}
	
	Lemma \ref{lem:2adic}, \eqref{equ:nudey}, Cauchy-Schwarz's inequality and the definition of $E$ imply
	\begin{align*}
	\delta^\epsilon\|\nude\|_2^2\ll \int_E \|\sum_{i,j}2^i\BB_{X_i}*2^j\BB_{yX_j}\|_2^2\frac{1}{|\det y|^2}\dd\nu(y)\ll (\log\delta)^2\delta^{-O(\epsilon)}\sum_{i,j}\int_E \|2^i\BB_{X_i}*2^j\BB_{yX_j}\|_2^2\dd\nu(y).
	\end{align*}
	There must exist $i,j$ such that
	\begin{equation}\label{equ:bbaiaj}
	Q:=\int_E \|2^i\BB_{X_i}*2^j\BB_{yX_j}\|_2^2\dd\nu(y)\gg \frac{\delta^{O(\epsilon)}}{(\log\delta)^4}\|\nude\|_2^2\gg\delta^{O(\epsilon)}\|\nude\|^2_2\geq \delta^{O(\epsilon)-\sigma_1}.
	\end{equation}
	
	With the same argument as in \cite[Appendix A.2]{boutonnet2017local}, we can conclude that $i,j>0$. If $i=0$, since $\supp \nu\subset[-\delta^{-\epsilon},\delta^{-\epsilon}]^n$, we have a bound on volume, that is, $|X_0|\leq \delta^{-O(\epsilon)}$. If $j>0$, by Lemma \ref{lem:2adic}, then $\|2^j\BB_{X_j}\|_1\ll \|\nu_{3\delta}\|_1=1$. Therefore, for $j\geq 0$ and $\|y\|\leq \delta^{-\epsilon}$, by Young's inequality
	\[ \|\BB_{X_0}*2^j\BB_{yX_j} \|_2\leq \|\BB_{X_0}\|_2\|2^j\BB_{yX_j} \|_1\leq \delta^{-O(\epsilon)},  \]
	which contradicts to \eqref{equ:bbaiaj} if $\epsilon$ is sufficiently small with respect to $\sigma$. Similarly, we obtain $j>0$.
	
	Therefore, Lemma \ref{lem:2adic} implies
	\begin{equation}\label{equ:aiaj}
	\begin{split}
	2^i|X_i&|=\|2^i\BB_{X_i}\|_1\ll \|\nu_{3\delta}\|_1=1,\\
	 2^{2i}|X_i|&=\|2^i\BB_{X_i}\|_2^2\ll \|\nu_{3\delta}\|_2^2\ll\|\nude\|_2^2, \text{ and similarly for }j,
	\end{split}
	\end{equation}
	where the last inequality is due to Lemma \ref{lem:nu4}.
	Hence by Young's inequality, for every $y$ in the support of $\nu$
	\begin{equation}\label{equ:2i2jleq}
		 \|2^i\BB_{X_i}*2^j\BB_{yX_j}\|_2\leq \|2^i\BB_{X_i}\|_1\|2^j\BB_{yX_j}\|_2=2^i|X_i|\|2^j\BB_{X_j}\|_2|\det y|^{1/2}\ll \delta^{-O(\epsilon)}\|\nude\|_2.
	\end{equation}
	
	Then we take a set $B$ such that for every $y$ in $B$ we have that $\|2^i\BB_{X_i}*2^j\BB_{yX_j}\|_2^2$ is relatively large. 
	Let
	\begin{equation}\label{equ:by}
	B=\{y\in E\cap\supp\nu|\ \|2^i\BB_{X_i}*2^j\BB_{yX_j}\|_2^2\geq Q/2 \}.
	\end{equation}
	Using Lemma \ref{lem: inverse cheb} with $f(y)=\|2^i\BB_{X_i}*2^j\BB_{yX_j}\|_2^2$ and \eqref{equ:bbaiaj}, \eqref{equ:2i2jleq} we have
	\begin{align}\label{equ:nuB}
	\nu(B)
	\geq  \frac{Q}{2\sup_{y\in\supp\nu}  f(y)}\gg \delta^{O(\epsilon)}.
	\end{align}
	
	We verify that $X_i,X_j$ and $B$ satisfy some natural assumptions. Take $y$ in $B$. By \eqref{equ:bbaiaj} and Young's inequality, we have
	\begin{equation}\label{equ:young}
	\begin{split}
	\delta^{O(\epsilon)}\|\nude\|_2\ll  \|2^i\BB_{X_i}*2^j\BB_{yX_j}\|_2\leq \|2^i\BB_{X_i}\|_2\|2^j\BB_{yX_j}\|_1&=2^j|X_j|\|2^i\BB_{X_i}\|_2|\det y|.
	\end{split}
	\end{equation}
	%By \eqref{equ:aiaj}, the inequality \eqref{equ:young} gives  
	%	\begin{equation}\label{equ:ygg}
	%	|\det y|\gg\delta^{O(\epsilon)}, \text{ for }y\in B.
	%	\end{equation}
	By $\|2^j\BB_{X_j}\|_2\ll \|\nude\|_2$ , $|\det y|\leq \delta^{-O(\epsilon)}$ and \eqref{equ:aiaj}, the inequality \eqref{equ:young} implies
	\begin{equation}\label{equ:2jaj}
	2^j|X_j|= \delta^{O(\epsilon)}\text{, and similarly }2^i|X_i|= \delta^{O(\epsilon)}.
	\end{equation}
	Next, \eqref{equ:aiaj} and \eqref{equ:young} also imply
	\begin{equation*}
		\delta^{O(\epsilon)}\|\nude\|_2\ll 2^j|X_j|\|2^i\BB_{X_i}\|_2|\det y|\ll \delta^{-O(\epsilon)}2^i|X_i|^{1/2}\leq \delta^{-O(\epsilon)}2^{i/2}.
	\end{equation*}
	We have
	\begin{equation}\label{equ:2igg}
		2^i\gg \delta^{O(\epsilon)}\|\nude\|_2^2\geq \delta^{-\sigma_1+O(\epsilon)}.
	\end{equation}
	Since $X_i$ is an essentially disjoint union of $\delta$ balls, we have $\ndelta(X_i)\sim\frac{|X_i|}{\delta^n}$ and $\ndelta(X_i\cap\pi_l^{-1}B_\RR(a,\rho))\ll\delta^{-n}|X_i\cap\pi_l^{-1}B_\RR(a,2\rho)|$ for every $\rho\geq\delta$ and $l=1,\cdots, n$ . By \eqref{equ:2jaj} and \eqref{equ:2igg} we have
	\begin{align}\label{equ:aill}
		&\ndelta(X_i)\sim\frac{|X_i|}{\delta^n}=\delta^{O(\epsilon)} 2^{-i}\delta^{-n}\ll \delta^{-n+\sigma_1-O(\epsilon)}.
	\end{align}
	By Lemma \ref{lem:2adic}(1), the projective non concentration and \eqref{equ:aill} for $\rho\geq\delta$, $a\in \R$ and $l=1,\cdots,n$,
	\begin{equation}\label{equ:aipij}
	\begin{split}
		\ndelta(X_i\cap\pi_l^{-1}B_\RR(a,\rho))&\ll\delta^{-n}|X_i\cap\pi_l^{-1}B_\RR(a,2\rho)|\leq \delta^{-n}2^{-i}\nu_{3\delta}(\pi_l^{-1}B_\RR(a,2\rho))\\
		&\ll \ndelta(X_i)\delta^{-O(\epsilon)}\rho^\kappa.
	\end{split}
	\end{equation}
	This means that $X_i$ inherits non concentration from $\nu$.
	
	We calculate additive energy. By \eqref{equ:additive energy} we have 
	\begin{align*}
	\omega_\delta(+,X_i\times yX_j)&\gg \delta^{-3n}\|\BB_{X_i}*\BB_{yX_j}\|_2^2.
	\end{align*}
	Then for every $y$ in $B$, by \eqref{equ:by}, \eqref{equ:aiaj}, \eqref{equ:2jaj} and \eqref{equ:aill}
	\begin{align*}
	\omega_\delta(+,X_i\times yX_j)&\gg\delta^{-3n+O(\epsilon)}\|\nude\|_2^2 2^{-2i-2j}\\
	&\gg \delta^{-3n+O(\epsilon)}2^{-i-j}|X_i|^{1/2}|X_j|^{1/2}\gg \delta^{-3n+O(\epsilon)}|X_i|^{3/2}|X_j|^{3/2}
	\\ &\gg \delta^{O(\epsilon)}\ndelta(X_i)^{3/2}\ndelta(X_j)^{3/2}.
	\end{align*}
	We can use the following proposition, which is a uniform version of the Balog-Szemerédi-Gowers theorem, inspired by the finite field version due to Bourgain.
	\begin{prop}[Additive-Multiplicative Balog-Szemerédi-Gowers theorem]\label{prop:BSG}
		Let $K>2$ be the roughness constant, let $X,X',B$ be bounded subsets of $\bb R^n$ in $B(0,K)$, with $B^{-1}$ bounded by $K$ (if $b\in B$ then $|b_j|\geq 1/K$ for $j=1,\dots n$), and let $\mu$ be a Borel probability measure on $B$. If for every $b\in B$ we have
		\[\omega_\delta(+,X\times bX')\geq\frac{1}{K}\ndelta(X)^{3/2}\ndelta(X')^{3/2}. \]
		Then there exist $X_o\subset X$, $b_o\in B$ and $B_1\subset B\cap B(b_o,1/K^2)$ containing $b_o$ such that $\ndelta(X_o)\geq K^{-O(1)}\ndelta(X)$, $\mu(B_1)\geq K^{-O(1)}$ and for every $b\in b_o^{-1}B_1$
		\[\ndelta(X_o+bX_o)\leq K^{O(1)}\ndelta(X_o). \]
	\end{prop}
	Take $K=\delta^{-O(\epsilon)}$, $\mu=\frac{1}{\nu(B)}\nu|_B$, $X=X_i$ and $X'=X_j$. By \eqref{equ:ygg}, the set $B$ satisfies the assumption in Proposition \ref{prop:BSG}. Take $B(1,2r)\subset U$ as in Theorem \ref{prop:sumpro} with the group $G=(\bb R^*)^n$, $V=\bb R^n$. Proposition \ref{prop:BSG} implies that  for $\delta$ small enough that $\delta^\epsilon\leq r$ there exist $C_1>0$,
	$$X_o\subset X_i \text{ and }B_1\subset B\cap B(b_o,\delta^\epsilon r)$$ 
	such that 
	\begin{equation}\label{equ:nx1ai}
		\ndelta(X_o)\geq \delta^{C_1\epsilon}\ndelta(X_i),
	\end{equation} 
	\begin{equation}\label{equ:mub0}
	\mu(B_1)\geq \delta^{C_1\epsilon}, \text{}
	\end{equation} 
	and for $b\in b_o^{-1}B_1$,
	\begin{equation}\label{equ:x1bx1}
	\ndelta(X_o+bX_o)\leq \delta^{-C_1\epsilon}\ndelta(X_o).
	\end{equation}
	\begin{lem}\label{lem:b1x1}
		There exists $C_2>0$. These sets $b_o^{-1}B_1,X_o$ satisfy the $(\delta,\kappa,\sigma_1,C_2\epsilon)$ assumption of Theorem \ref{prop:sumpro} when $\delta$ is small enough.
	\end{lem}
	\begin{proof}
		By Proposition \ref{prop:BSG}, the set $X_o$ satisfies $X_o\subset X_i\subset  \supp\nu^{(4\delta)}\subset B(0,\delta^{-O(\epsilon)})$, and $B_1$ satisfies $b_o^{-1}B_1\subset b_o^{-1}B(b_o,\delta^\epsilon r)\subset U$.
		
		Let 
		\begin{equation*}
		\nu_1=\frac{1}{\nu(B_1)}(\nu|_{B_1}).
		\end{equation*}
		By \eqref{equ:nuB} and \eqref{equ:mub0}
		\[\nu(B_1)=\nu(B)\mu(B_1)\gg \delta^{O(\epsilon)}. \] 
		Hence for any Borel measurable set $E$, we have 
		\begin{equation}\label{equ:nuone leq nu}
		\nu_1(E)\leq \delta^{-O(\epsilon)}\nu(E).
		\end{equation}
		
		Assumption (i) (non concentration): By \eqref{equ:nuone leq nu} and projective non concentration
		\begin{equation*}
		\forall \rho>\delta,\ \sup_{a\in\bb R}\nu_1(\pi_j^{-1} B_\RR(a,\rho) )\ll \delta^{-O(\epsilon)}\sup_{a\in\bb R}\nu(\pi_j^{-1} B_\RR(a,\rho) ) \leq \delta^{-O(\epsilon)}\rho^\kappa,
		\end{equation*}
		Therefore by $\|b_o^{-1}\|\leq \delta^{-O(\epsilon)}$ and Lemma \ref{lem:klip}, 
		\begin{equation}
		\nrho(\pi_j(b_o^{-1}B_1))\geq \delta^{O(\epsilon)}\nrho(\pi_j(B_1))\geq \frac{\nu_1(B_1)}{\sup_{a\in\bb R}\nu_1(\pi_j^{-1} B_\RR(a,\rho) )} \geq \delta^{O(\epsilon)}\rho^{-\kappa}. 
		\end{equation}
		
		Assumption (ii) (away from proper unitary subalgebras): All the maximal unitary subalgebras of $\RR^n$ have a form $\{x\in\RR^n| x^i=x^j \}$ with $i\neq j$. Let $f_{ij}(x)=x^i-x^j$ for $x\in\R^n$. By \eqref{equ:ygg} we know that $|(b_o)_i|,|(b_o)_j|\geq \delta^{O(\epsilon)}$. By \eqref{equ:nuone leq nu},
		$$\nu_1\{x| f_{ij}(b_o^{-1}x)\in B_\R(0,\rho) \}\leq \delta^{-O(\epsilon)}\nu\{x| f_{ij}(b_o^{-1}x)\in B_\R(0,\rho) \}.$$
		This is an estimate of being away from linear subspace. If we take the vector $w$ with its $i$-th, $j$-th coordinate equal to $(b_o)_i^{-1}$, $-(b_o)_j^{-1}$, and other coordinates equal to zero, and let $v=w/\|w\|$, then
		\[f_{ij}(b_o^{-1}x)=\l w,x\r. \]
		 Hence projective non concentration \eqref{equ:procon} for $v$  implies that
		$$ \nu\{x| f_{ij}(b_o^{-1}x)\in B_\R(0,\rho) \}\leq\nu(\pi_v^{-1} B_\R(0,\delta^{-O(\epsilon)}\rho))\leq  \delta^{-O(\epsilon)}\rho^\kappa.$$ 
		Hence $b_o^{-1}B_1$ is $\delta^{O(\epsilon)}$ away from proper subalgebra.
		
		Assumption (iii) (non concentration of $X_o$): By  \eqref{equ:nx1ai}  and \eqref{equ:aipij} we have for $\rho\geq\delta$ and $j=1,\cdots, n$,
		\[\nrho(\pi_j(X_o))\geq\frac{\ndelta(X_o)}{\sup_{a\in\bb R}\ndelta(X_o\cap\pi_j^{-1}B_\R(a,\rho))}\gg \delta^{O(\epsilon)}\frac{\ndelta(X_i)}{\sup_{a\in\bb R}\ndelta(X_i\cap\pi_j^{-1}B_\R(a,\rho))}\gg \delta^{O(\epsilon)}\rho^{-\kappa}. \]
		
		Assumption (iv): By \eqref{equ:aill}, 
		\[\ndelta(X_o)\ll\ndelta(X_i)\ll \delta^{-n+\sigma_1-O(\epsilon)}. \]
		When $\delta$ is small enough such that $\delta^\epsilon\leq 1/2$, the inequalities with Landau notation can be replaced by $\geq$ or $\leq$ with augmenting $O(\epsilon)$.
	\end{proof}
	The end of the proof of the $L^2$-flattening lemma:
	Let $C_1\epsilon$ and $C_2\epsilon$ be given in \eqref{equ:x1bx1} and Lemma \ref{lem:b1x1}, respectively. Suppose that $C_2\geq C_1$ (we can always augment $C_2$ in Lemma \ref{lem:b1x1}. The larger $C_2$ is, the easier the assumption is). Applying Theorem \ref{prop:sumpro} with $A=b_o^{-1}B_1\supset \{\id \}$ and $X=X_o$, when $\epsilon$ is sufficiently small, we have
	\[\sup_{b\in b_o^{-1}B_1}\ndelta(X_o+bX_o)\geq\delta^{-C_2\epsilon}\ndelta(X_o) .\]
	Due to $C_2\geq C_1$, we have $\delta^{-C_2\epsilon}\ndelta(X_o)\geq \delta^{-C_1\epsilon}\ndelta(X_o)$, which contradicts \eqref{equ:x1bx1}.
	The proof is complete.
\end{proof}
\begin{rem}\label{rem:weak}
	The only place where we need a stronger non concentration than non concentration on coordinate subspaces is in the proof of Lemma \ref{lem:b1x1}, when we check assumption (ii) of Theorem \ref{prop:sumpro}. In this step, we need a property of being away from a linear subspace.
\end{rem}
It remains to prove Proposition \ref{prop:BSG}. We first state the finite field version
\begin{prop}\cite[Thm.C]{bourgain_multilinear_2009} \cite[Prop. 4.1]{green_sum-product_2009}
	Let $K>1$. Let $A\subset\bb F_p$ and $B\subset\bb F_p^*$ be two sets. If for all $b$ in $B$, we have $\omega(+,A\times bA)\geq K^{-1}|A|^{3/2}|B|^{3/2}$. Then there exist $x$ in $B$ and $A'\subset A$, $B'\subset x^{-1}B$ with $|A'|\geq K^{-O(1)}|A|$ and $|B'|\geq K^{-O(1)}|B|$ such that for all $b'\in B'$,
	\[|A'+b'A'|\leq K^{O(1)}|A'|. \]
\end{prop}
The main point is to find $A'$ which is uniform for $b$. This is accomplished by using the pigeon-hole principle. For more details, please see \cite[Prop. 4.1]{green_sum-product_2009} or the following proof.
\begin{proof}[Proof of Proposition \ref{prop:BSG}]
	We follow closely the proof of \cite[Proposition 4.1]{green_sum-product_2009}. Since $B$ and $B^{-1}$ are bounded by $K$, if we multiply a set by an element in $B$, then Lemma \ref{lem:klip} implies that we only lose some power on $K$, which does not change the result. That means for $b$ in $B$ and a subset $X$ of $\R^n$, we have
	\[K^{-O(1)}\ndelta(bX)\leq \ndelta(X)\leq K^{O(1)}\ndelta(bX).\]
	Hence, we will not write the comparison of $\ndelta(A)$ with $\ndelta(bA)$ for bounded set $A$. They have the same size.
	
	For every $b\in B$, using additive Balog-Szemerédi-Gowers theorem (Proposition \ref{prop:bsg}), we have $X_b\times X_b'\subset X\times X'$ such that 
	\begin{equation}\label{equ:xbbyb}
	\ndelta(X_b+bX_b')\leq K^{O(1)}\ndelta(X)^{1/2}\ndelta(X')^{1/2}
	\end{equation}
	 and 
	 \begin{equation}\label{equ:xbx}
	 \ndelta(X_b)\geq K^{-O(1)}\ndelta(X),\ \ndelta(X_b')\geq K^{-O(1)}\ndelta(X').
	 \end{equation}
	  The result we need is a uniform version, independent of $b$. For this purpose, 
	we want to find an element $b_o$ in $B$ and a portion of $B$ such that the intersection of two sets $X_{b_o}$ and $X_b$ is large for $b$ in this portion. 
	\begin{lem}\label{lem:bb'}
		Let $\mu$ be a probability measure on a set $B\subset B_{\R^n}(K)$. Let $S$ be a compact set of $\R^n$. Assume that for every $b$ in $B$, there exists $S_b\subset S$ such that
		\[ |S_b|\geq K^{-1}|S|. \]
		Then there exist
			$b_o$ in $B$ and $B_1\subset B\cap B(b_o,1/K^2)$ containing $b_o$ such that $\mu(B_1)\geq K^{-O(1)}$, and for every $b$ in $B_1$
			\begin{equation}\label{equ:bb'}
					|S_b\cap S_{b_o}|\geq K^{-O(1)}|S|. 
			\end{equation}
	\end{lem}
		\begin{proof}

		We cover $B$ with $O(K^{3n})$ balls of radius $1/K^2$, written as $C_1,\dots, C_j$. We claim that:
		There exists $i$ such that
		\begin{equation}\label{equ:ci2}
		\int_{C_i^2}|S_b\cap S_{b'}|\dd\mu(b)\dd\mu(b')\gg K^{-O(1)}|S|.
		\end{equation}
		By hypothesis, we have
		\begin{equation}\label{equ:BS}	
			\int_{B}\int_{S}\BB_{S_b}(x)\dd x\dd\mu(b)=\int_B|S_b| \dd\mu(b)\geq K^{-1}|S|.
		\end{equation}
			
		By Cauchy-Schwarz's inequality
		\begin{align}\label{equ:orici}
		K^{3n}\sum_{i}\left(\int_{C_i} \BB_{S_b}(x)\dd\mu(b)\right)^2\gg  \left(\int_B \BB_{S_b}(x)\dd\mu(b)\right)^2.
		\end{align}
		By Cauchy-Schwarz's inequality and \eqref{equ:BS}
		\begin{equation}\label{equ:sbsb}
		\int_S\left(\int_B \BB_{S_b}(x)\dd\mu(b)\right)^2\dd x\geq \left(\int_S\int_B\BB_{S_b}(x)\dd\mu(b)\dd x\right)^2/|S|\geq K^{-O(1)}|S|.
		\end{equation}
		Rewrite the left hand side of \eqref{equ:orici} and integrate it with respect to the Lebesgue measure on $S$.
		Combined with \eqref{equ:sbsb} we have
		\[\sum _i\int_{C_i^2}|S_b\cap S_{b'}|\dd\mu(b)\dd\mu(b')\gg K^{-O(1)}|S|.  \]
		The claim \eqref{equ:ci2} follows.

		 By Lemma \ref{lem: inverse cheb}, we can find $C'$, a subset of $C_i^2$ and containing $\{(b,b)| b\in C_i \}$., such that $\mu\otimes\mu(C')\gg K^{-O(1)}$ and for all $(b,b')\in C'$
		\begin{equation}\label{equ:Yb cap Yb}
		|S_b\cap S_{b'}|\gg K^{-O(1)}|S|. 
		\end{equation}
		By Fubini's theorem, we can find a $b_o$ such that $\mu\{b\in C_i|(b_o,b)\in C' \}\gg K^{-O(1)}$. We let $B_1=\{b\in C_i|(b_o,b)\in C' \}$, then this set satisfies the measure assumption.
	\end{proof}
		The $\delta$ neighbourhood of a set behaves well under intersection. In order to simplify the notation, abbreviate $X^{(\delta)}, X'^{(\delta)}, X_b^{(\delta)},X_b'^{(\delta)}$ to $Y,Y',Y_b,Y_b'$. By \eqref{equ:neiequ} we have
		\begin{equation}\label{equ:delta neighborhood}
		\ndelta(X)\sim |Y|\delta^{-n}.
		\end{equation}
		Due to \eqref{equ:xbbyb} and \eqref{equ:xbx}, we have $\ndelta(X)^{1/2}\ndelta(X')^{1/2}\geq K^{-O(1)}\ndelta(X_b+bX_b')\geq K^{-O(1)}\ndelta(X_b)\geq K^{-O(1)}\ndelta(X)$, which implies
		\begin{equation*}
		\ndelta(X)\sim_{K^{O(1)}}\ndelta(X').
		\end{equation*}
		Hence
		\begin{equation}\label{equ:Y K Y'}
		|Y|\sim \delta^n\ndelta(X)\sim_{K^{O(1)}}\delta^n\ndelta(X') \sim |Y'|.
		\end{equation}
		Let
			\begin{equation}
			S=Y\times Y' \text{ and }S_b=Y_b\times Y_b' \text{ for }b\in B.
			\end{equation}
		By \eqref{equ:xbx}, we have $|Y_b|\geq \delta^{O(\epsilon)}|Y|$ and $|Y_b'|\geq\delta^{O(\epsilon)}|Y'|$. Hence, we can use Lemma \ref{lem:bb'} with $K=\delta^{-O(\epsilon)}$ to obtain an element $b_o$ and a set $B_1$ with desired property.  Next, we want to find $X_o$. Due to 
		\[\delta^{O(\epsilon)}|Y||Y'|=\delta^{O(\epsilon)}|S|\leq |S_b\cap S_{b_o}|=|Y_b\cap Y_{b_o}||Y_b'\cap Y_{b_o}'|, \]
		together with \eqref{equ:Y K Y'}, we obtain
		\begin{equation}\label{equ:bb0}
			|Y_b\cap Y_{b_o}|, |Y_b'\cap Y_{b_o}'|\geq \delta^{O(\epsilon)}|Y|.
		\end{equation}
	
	The proof concludes by Ruzsa calculus. 
	  By Lemma \ref{lem:ruzsa}(1) and \eqref{equ:xbbyb}, we have $$\sigma_\delta[X_{b_o}],\sigma_\delta[X_b],\sigma_\delta[X_{b_o}'],\sigma_\delta[X_b']\leq K^{O(1)}.$$ 
	  By \eqref{equ:bb0} and \eqref{equ:delta neighborhood}, we have $$|X_{b_o}^{(\delta)}\cap X_{b}^{(\delta)}|,|X_{b_o}'^{(\delta)}\cap X_{b}'^{(\delta)}|\geq K^{-O(1)}|X^{(\delta)}|\geq K^{-O(1)}\delta^n\ndelta(X).$$
	By Lemma \ref{lem:ruzsa}(3), we have $X_{b_o}\approx_{K^{O(1)}} X_b$ and $X_{b_o}'\approx_{K^{O(1)}} X_b'$, the latter implies $bX_{b_o}'\approx_{K^{O(1)}} bX_b'$. Therefore by Lemma \ref{lem:ruzsa}(2)
	\[X_{b_o}\approx_{K^{O(1)}} X_b\approx_{K^{O(1)}} bX_b'\approx_{K^{O(1)}} bX_{b_o}'=\frac{b}{b_o}b_oX_{b_o}'\approx_{K^{O(1)}}\frac{b}{b_o}X_{b_o}. \]
	We get $X_{b_o}\approx_{K^{O(1)}} \frac{b}{b_o}X_{b_o}$. Let $X_o=X_{b_o}\subset X$. The proof is complete.
\end{proof}

\subsection{Proof of the Fourier decay of multiplicative convolutions}
\label{sec:foudec}
 Using $L^2$-flattening (Lemma \ref{lem:l2fla}), we give a proof of Theorem \ref{thm:sumfourier}. The strategy is to apply $L^2$-flattening to 
 \[\nu=\frac{1}{2}\left(\murr{2k}+\murr{k}\right). \]
 We need a lemma which explains the connection between $\|\nude\|_2$ and the Fourier transform of $\nu$.
 	\begin{lem}\label{lem:nudel}
 		Let $\delta>0$, $C>8$ and let $\delta_1=2\delta/C$. Let $\nu$ be a Borel probability measure on $\RR^n$ with  support in $B(0,C)$. We have
 		\begin{equation}\label{equ:nudel}
 		\|\nu_{\delta_1}\|_2^2\sim_C \int_{B(0,2\delta^{-1})}|\hat{\nu}(\xi)|^2\dd\xi,
 		\end{equation}
 		\begin{equation}\label{equ:nudelcon}
 		\int \|\nu_{\delta_1}*(m_y)_*\nu_{\delta_1}\|_2^2\dd\nu(y)\gg\int_{B(0,2\delta^{-1})}\int|\hat{\nu}(\xi)|^2|\hat{\nu}(y\xi)|^2\dd\nu(y)\dd\xi.
 		\end{equation}
 	\end{lem}
 	The proof of Lemma \ref{lem:nudel} will be given at the end of this section.

  Recall that $\mu_k$ is the $k$-times multiplicative convolution of $\mu$. We have
  \[\left|\int \exp(2i\pi\l \xi, x_1\cdots x_k\r)\dd\mu(x_1)\cdots \dd\mu(x_k)\right|=|\hat\mu_k(\xi)|. \]
  For $k,r\in\bb N$, let $\sigma_{k,r}$ be the real number defined by
  \begin{equation}
  \sigma_{k,r}=\frac{\log\int_{\xi\in B(0,2\delta^{-1})}|\hat{\mu}_k(\xi)|^{2r}\dd\xi}{|\log \delta|}\sim\frac{\log\|(\mu_{k,r})_\delta\|_2^2}{|\log \delta|},
  \end{equation}
  where $\mu_{k,r}=(\mu_k*\mu_k^{-})^{*r}$. 
  
  The remainder of the proof is to control $\sigma_{k,r}$, divided into two steps. We first prove that if $\sigma_{k,r}$ is not sufficiently small, then $L^2$-flattening (Lemma \ref{lem:l2fla}) reduces the value of $\sigma_{k,r}$. When $\sigma_{k,r}$ is sufficiently small, the H\"older regularity of $\mu$ enables us to finish the proof.  This can be understood that if a measure $\mu$ satisfies non concentration assumption, then after sufficient multiplicative and additive convolutions, the sum-product phenomenon implies that $\mu_{k,r}$ is much more smooth.
  
\begin{proof}[Proof of Theorem \ref{thm:sumfourier}]
	Let
	\begin{equation}\label{equ:epskap}
		\kappa_1=\kappa_0/4, \ \epsilon=\min\{ \epsilon(\kappa_1/2,\kappa_0),\kappa_0\}/2,
	\end{equation}
	where $\epsilon(\kappa_1/2,\kappa_0)$ is given in $L^2$-flattening (Lemma \ref{lem:l2fla}).
	
	\textbf{Reducing the value}:
	We have a consequence of $L^2$-flattening (Lemma \ref{lem:l2fla}), whose proof will be given later.
	\begin{lem}\label{lem2kk}
		Under the assumption of Theorem \ref{thm:sumfourier}, if $\sigma_{k,r}\geq\kappa_1$, then for $\delta$ small enough depending on $k,r$, we have
		\[\sigma_{2k,r'}\leq \sigma_{k,r}-\epsilon, \]
		where $r'=8r^2+4r$.
	\end{lem}
	
	\textbf{Sufficient regularity}: We have a higher dimensional version of \cite[Theorem 7]{bourgain2010discretized}, which says that if two measures have sufficient H\"older regularity, then the multiplicative convolution of these two measures has power decay in its Fourier transform.
	\begin{lem}\label{lem:alpbet} Let $\alpha>\beta>0$ and $\delta>0$. Let $\mu$ be a measure on $B(0,1)$ such that for $j=1,\dots,n$
		\begin{equation}\label{equ:supmua}
		\sup_{a}(\pi_j)_*\mu(B_\R(a,\delta))\leq \delta^{\alpha}.
		\end{equation}
		Let $K>2$ be a parameter. Let $\nu$ be a compactly supported measure on $B(0,K)$ such that
		\begin{equation}\label{equ:nvfoudec}
		\int_{B(0,2\delta^{-1})} |\hat{\nu}(\xi)|\dd\xi\leq \delta^{-\beta}.
		\end{equation}
		Then for $\|\xi\|\in [\delta^{-1}/2,\delta^{-1}]$
		\[\int |\hat{\nu}(x\xi)|\dd\mu(x)\ll_{K,n} \delta^{\frac{\alpha-\beta}{n+2}}. \] 
	\end{lem}
	The proof of Lemma \ref{lem:alpbet} is classic and will be given at the end of Section \ref{sec:foudec} for completeness.

	If $\sigma_{1,1}\geq \kappa_1$, iterating Lemma \ref{lem2kk} several times implies that $\sigma_{k,r}<\kappa_1$, where $k,r$ only depend on $\kappa_1$.
	 
	We will now apply Lemma \ref{lem:alpbet} to a well-chosen measure. Take $(\mu_k*\mu_k^-)^{(r)}$ as $\nu$, $\alpha=\kappa_0-\epsilon$, $\beta=\kappa_1$ and $\tau=\frac{\alpha-\beta}{n+2}$. For $\|\xi\|\in[\delta^{-1}/2,\delta^{-1}]$, by H\"older's inequality and Lemma \ref{lem:alpbet},
	\begin{align*}
	|\hat\mu_{k+1}(\xi)|^{2r}=\left|\int\hat\mu_k(x\xi)\dd\mu(x)\right|^{2r}\leq\int|\hat\mu_k(x\xi)|^{2r}\dd\mu(x)=\int|\hat{\nu}(x\xi)|\dd\mu(x)\leq_{k,n} \delta^{\tau}. 
	\end{align*}
	When $\delta$ is small enough, this yields \eqref{equ:exp2pi} with $$\epsilon_1=\frac{\tau}{4r}=\frac{\kappa_0-\epsilon-\kappa_1}{4(n+2)r}\geq\frac{\kappa_0/2-\kappa_1}{4(n+2)r}\geq \frac{\kappa_1}{4(n+2)r} ,$$ 
	where the last two inequalities are due to \eqref{equ:epskap} and $r$ only depends on $\kappa_0$.
\end{proof}
Now we will prove Lemma \ref{lem2kk}, where we use the $L^2$-flattening (Lemma \ref{lem:l2fla}). 
\begin{proof}[Proof of Lemma \ref{lem2kk}]
	Fix $k,r$ and set
	\begin{equation}\label{equ:nu}
	\nu=\frac{1}{2}\left(\murr{2k}+\murr{k}\right).
	\end{equation} 
	This is the key construction of this proof. The measure $\nu$ is a bridge connecting $\mu_{2k}$ to $\mu_k$. We summarize the properties of $\nu$ in the following lemma.
	\begin{lem}\label{lem:nucon}
		 The measure $\nu$ satisfies $(\delta/r,\kappa_0,2\epsilon)$ projective non concentration assumption when $\delta$ is sufficiently small depending on $k,r$.
	\end{lem}
\begin{proof}
	Projective non concentration property is preserved under additive convolution. that is, if a probability measure $m$ satisfies projective non concentration, then $m*m'$ also satisfies projective non concentration for any probability measure $m'$. The reason is the following calculation. By Fubini's theorem, we have
	\[(\pi_v)_*(m*m')(B_{\R}(a,\rho))\leq \sup_{b\in\R} (\pi_v)_*m(B_{\R}(b,\rho)).\]
	
	 Hence we can drop the additive convolution, and for $\rho\geq\delta_1$, we have
	\begin{equation}\label{equ:anub}
	\sup_{a\in\RR,v\in\bb S^{n-1}}(\pi_v)_*\nu( B_\RR(a,\rho))\leq \frac{1}{2}\sup_{a,v}(\pi_v)_*\mu_k(B_\RR(a,\rho))+\frac{1}{2}\sup_{a,v}(\pi_v)_*\mu_{2k}(B_\RR(a,\rho)).
	\end{equation}
	The property that the support of $\mu$ is contained in $[1/2,1]^n$ and the projective non concentration of $\mu$ imply the left hand side of \eqref{equ:anub} is less than
	\begin{equation}\label{equ:mub}
	\sup_{a,v}(\pi_v)_*\mu(B_\RR(a,4^k\rho)) 
	\leq\delta^{-\epsilon} (\max\{4^k\rho,r\rho\})^{\kappa_0}\leq \delta_1^{-2\epsilon}\rho^{\kappa_0},
	\end{equation}
	where we have used $r\rho\geq r\delta_1=\delta$ for projective non concentration and the last inequality holds for $\delta$ small enough depending on $k,r$. Then \eqref{equ:procon} follows from \eqref{equ:anub} and \eqref{equ:mub}.
	The measure $\nu$ satisfies non concentration with $(\kappa_0,2\epsilon)$ at scale $\delta_1$. 
\end{proof}
\begin{rem}\label{rem:pro}
	This is a step where we really need projective non concentration.
\end{rem}
\begin{lem}\label{lem:bonmu}
	 Let $C>0$ and $r\in\bb N$. Let $\mu$ be a probability measure on $[1/C,1]^n\subset\R^n$. Let $\nu$ be defined by
	 \[\nu=\frac{1}{2}\left(\murr{2}+\murr{}\right).\]
	 We have
	 \begin{equation}\label{equ:bonmu}
	 \int_{B(0,2\delta^{-1})}|\hat{\nu}(\xi)|^2\dd\xi\sim_C\int_{B(0,2\delta^{-1})}|\hat{\mu}(\xi)|^{4r}\dd\xi,%=\delta^{-\sigma_{k,2r}},
	 \end{equation}
	 and
	 \begin{equation}\label{equ:bonmunu}
	 \int_{B(0,2\delta^{-1})}|\hat{\mu}_{2}(\xi)|^{r'}\dd\xi
	 \ll\int_{B(0,2\delta^{-1})}\int|\hat{\nu}(\xi)|^2|\hat{\nu}(y\xi)|^2\dd\nu(y)\dd\xi,
	 \end{equation}
	where $r'=8r^2+4r$.
\end{lem}
The proof is an elementary computation, using Fourier transform and the H\"older inequality.
\begin{proof}
	The lower bound part of \eqref{equ:bonmu} is trivial, which is due to the definition of $\nu$.
	
	For a measure $m$ on $\bb R$ and $r\in\bb N$, we have a formula 
	\begin{equation}\label{equ:hatmu}
	|\hat{m}(\xi)|^{4r}=|\widehat{(m*m^-})^{(2r)}(\xi)|.
	\end{equation}
	By the multiplicative structure of $\bb R^n$, we have
	\begin{equation}\label{equ:hatmu2}
	|\hat{\mu}_{2}(\xi)|=\left|\int e^{2i\pi \l \xi, xy\r}\dd\mu_{}(x)\dd\mu(y)\right|=\left|\int\hat{\mu}(y\xi)\dd\mu(y)\right|.
	\end{equation}
	By the H\"older inequality,
	\begin{equation*}
		|\hat{\mu}_{2}(\xi)|^{4r}\leq \int|\hat{\mu}(y\xi)|^{4r}\dd\mu(y).
	\end{equation*}
	Integrating $\xi$ on $B(0,2\delta^{-1})$, we have
	\begin{equation*}
	\int_{B(0,2\delta^{-1})}|\hat{\mu}_{2}(\xi)|^{4r}\dd\xi\leq \int_{y\in\R^n,\xi\in B(0,2\delta^{-1})}|\hat{\mu}(y\xi)|^{4r}\dd\mu(y)\dd\xi. 
	\end{equation*}
	Due to $\supp\mu\subset[1/C,1]^n$, we have
	\begin{equation*}
	\begin{split}
	\int_{B(0,2\delta^{-1})}\int|\hat{\mu}(y\xi)|^{4r}\dd\mu(y)\dd\xi&\leq C^{n} \int\int_{B(0,2\delta^{-1})}|\hat{\mu}(y\xi)|^{4r}\dd( y\xi)\dd\mu(y)\\
	&= C^{n}\int_{B(0,2\delta^{-1})}|\hat{\mu}(\xi)|^{4r}\dd\xi,
	\end{split}
	\end{equation*}
	which implies that 
	\[	\int_{B(0,2\delta^{-1})}|\hat{\mu}_{2}(\xi)|^{4r}\dd\xi\ll_C\int_{B(0,2\delta^{-1})}|\hat{\mu}(\xi)|^{4r}\dd\xi. \]
	Therefore \eqref{equ:bonmu} follows from
	\begin{equation*}
	\int_{B(0,2\delta^{-1})}|\hat{\nu}(\xi)|^2\dd\xi=\frac{1}{4}\int_{B(0,2\delta^{-1})}\left(|\hat\mu(\xi)|^{2r}+|\hat\mu_{2}(\xi)|^{2r} \right)^2\dd\xi \ll_C\int_{B(0,2\delta^{-1})}|\hat{\mu}(\xi)|^{4r}\dd\xi.
	\end{equation*}	
	
	By \eqref{equ:hatmu2}, H\"older's inequality and \eqref{equ:hatmu}
	\begin{align*}
	|\hat{\mu}_{2}(\xi)|^{8r^2}&=\left|\int\hat{\mu}(x\xi)\dd\mu(x)\right|^{8r^2}\leq\left(\int|\hat{\mu}(x\xi)|^{2r}\dd\mu(x)\right)^{4r}
	\\
	&=\left|\int\widehat{(\mu*\mu^-)^{(r)}}(x\xi)\dd\mu(x)\right|^{4r}.
	\end{align*}
	By the Plancherel theorem and H\"older's inequality, the above inequality becomes
	\begin{equation}\label{equ:hatm}
	|\hat{\mu}_{2}(\xi)|^{8r^2}\leq \left|\int\hat{\mu}(y\xi)\dd\murr{}(y)\right|^{4r}\\
	\leq\int|\hat{\mu}(y\xi)|^{4r}\dd\murr{}(y).
	\end{equation}
	Let 
	\begin{equation*}
	A_{r}=\int_{y\in\R^n,\xi\in B(0,2\delta^{-1})}|\hat{\mu}(\xi)|^{4r}|\hat{\mu}(y\xi)|^{4r}\dd\murr{}(y)\dd\xi.
	\end{equation*}
	Therefore, by \eqref{equ:hatm} and \eqref{equ:hatmu}
	\begin{equation}\label{equ:bonk}
	\begin{split}
	\int_{B(0,2\delta^{-1})}|\hat{\mu}_{2}(\xi)|^{8r^2+4r}\dd\xi\leq\int_{B(0,2\delta^{-1})}|\hat{\mu}_{2}(\xi)|^{4r}\int|\hat{\mu}(y\xi)|^{4r}\dd\murr{}(y)\dd\xi=
	A_{r}.
	\end{split}
	\end{equation}	
	By \eqref{equ:nu}
	\begin{equation*}
	A_{r}
	\ll\int_{y\in\R^n,\xi\in B(0,2\delta^{-1})}|\hat{\nu}(\xi)|^2|\hat{\nu}(y\xi)|^2\dd\nu(y)\dd\xi.
	\end{equation*}
	Combined with \eqref{equ:bonk}, we obtain \eqref{equ:bonmunu}.
\end{proof}
	Lemma \ref{lem:nucon} and Lemma \ref{lem:bonmu} enable us to decrease the parameter $\sigma_{k,r}$ by $L^2$-flattening (Lemma \ref{lem:l2fla}). 
	
	We return to the proof of Lemma \ref{lem2kk}. By \eqref{equ:bonmu} and the hypothesis $\sigma_{k,2r}\geq\kappa_1$, we have
	\begin{equation}\label{equ:hatnu}
	\int_{B(0,2\delta^{-1})}|\hat \nu(\xi)|^2\dd\xi\gg_k \delta^{-\kappa_1}.
	\end{equation}
	 Due to $\supp\nu\in[-r,r]^n$ and \eqref{equ:hatnu}, taking $C=r$ in Lemma \ref{lem:nudel}, we have
	 \[\|\nu_{\delta_1}\|_2^2\gg_{r,k} \delta^{-\kappa_1}=r^{-\kappa_1}\delta_1^{-\kappa_1}. \]
	 When $\delta$ is small enough depending on $k,r,\kappa_1$, we have
	 \begin{equation}
	 	\|\nu_{\delta_1}\|^2_2\geq \delta_1^{-\kappa_1/2} \text{ and }\supp\nu\subset [-r,r]^n\subset [-\delta_1^{-2\epsilon},\delta_1^{-2\epsilon}]^n.
	 \end{equation}
	 Lemma \ref{lem:nucon} implies that $\nu$ satisfies assumption of $L^2$-flattening lemma with $\sigma_1=\kappa_1/2,\kappa=\kappa_0$ at scale $\delta_1$. Also notice that \eqref{equ:epskap} implies $2\epsilon\leq \epsilon(\kappa_1/2,\kappa_0)$. Then $L^2$-flattening (Lemma \ref{lem:l2fla}) implies
	 \begin{equation}\label{equ:l2fla}
	 	\int \|\nu_{\delta_1}*(m_y)_*\nu_{\delta_1}\|_2^2\dd\nu(y)\leq \delta_1^{2\epsilon}\|\nu_{\delta_1}\|^2_2.
	 \end{equation}
	 Using Lemma \ref{lem:nucon}, we obtain
	 \begin{equation}\label{equ:l2falnu}
	 	\int_{B(0,2\delta^{-1})}\int|\hat{\nu}(\xi)|^2|\hat{\nu}(y\xi)|^2\dd\nu(y)\dd\xi\leq_r \delta_1^{2\epsilon} \int_{B(0,2\delta^{-1})}|\hat{\nu}(\xi)|^2\dd\xi.
	 \end{equation}
	 Using Lemma \ref{lem:bonmu} with $\mu=\mu_k$ and $C=2^k$, by \eqref{equ:l2falnu}, we have
	\begin{align*}
	\int_{B(0,2\delta^{-1})}|\hat{\mu}_{2k}(\xi)|^{r'}\dd\xi
	\ll_{r,k}\delta_1^{2\epsilon} \int_{B(0,2\delta^{-1})}|\hat{\mu}_{k}(\xi)|^{4r}\dd\xi = \delta^{-\sigma_{k,2r}}\delta_1^{2\epsilon}\ll_r \delta^{2\epsilon-\sigma_{k,2r}}. 
	\end{align*}
	Therefore we have 
	$$\sigma_{2k,r'}\leq \sigma_{k,2r}-2\epsilon+C_{k,r}/\log \delta^{-1},$$
	with some constant $C_{k,r}>0$. For $\delta$ small enough, it follows that $\delta_{2k,r'}\leq\sigma_{k,2r}-\epsilon$.
\end{proof}
It remains to prove Lemma \ref{lem:nudel} and Lemma \ref{lem:alpbet}.

%  change notation $\delta_1,\delta_2,\delta$
\begin{proof}[Proof of Lemma \ref{lem:nudel}]
	Recall that $\delta=2C\delta_1$. We observe that the Fourier transform of $P_\delta$ satisfies
	\[\widehat{P_\delta}(\xi)=\int P_\delta(x)e^{2\pi i\l \xi, x\r}\dd x=\int P_1(x/\delta)\delta^{-n}e^{2\pi i\l \xi, x\r}\dd x=\widehat{P_1}(\delta\xi). \]
	Due to $\widehat{P_1}(\xi)=\Re \int_{B(0,1)} e^{2\pi i\l \xi, x\r}\dd x |B(0,1)|^{-1}=\int_{B(0,1)}\cos (2\pi\l \xi, x\r)\dd x|B(0,1)|^{-1}$, we see that
	\begin{equation}\label{equ:pos}
		\widehat{P_1} \text{ is real, positive and bounded away from zero for }\xi \text{ in } B(0,1/8).
	\end{equation}
	
	We are going to prove \eqref{equ:nudel}. By \eqref{equ:pos}, we have $\widehat{P_1}(\delta_1\xi)\gg 1$ for $\xi$ in $B(0,1/\delta_1)$, which implies
	\begin{equation}
	\begin{split}
				\|\nu_{\delta_1}\|^2_2&=\int |\hat{\nu}(\xi)|^2|\widehat{P_{\delta_1}}(\xi)|^2\dd\xi=\int |\hat{\nu}(\xi)|^2|\widehat{P_1}(\delta_1\xi)|^2\dd\xi\\
				&\gg \int_{B(0,1/(8\delta_1))} |\hat{\nu}(\xi)|^2\dd\xi\geq\int_{B(0,2\delta^{-1})}|\hat{\nu}(\xi)|^2\dd\xi.
	\end{split}
	\end{equation}
	For the other direction of \eqref{equ:nudel}, let $\delta_2=2\delta=4C\delta_1$. Due to $1/\delta_2+1/\delta_2=1/\delta$, we have $P_{1/\delta}\gg P_{1/\delta_2}*P_{1/\delta_2}$, which implies
	\begin{equation}\label{equ:bo2}
	\begin{split}
	\int_{B(0,2\delta^{-1})} |\hat{\nu}(\xi)|^2\dd\xi&\gg\int|\hat{\nu}(\xi)|^2|P_{1/\delta}(\xi)|^2\delta^{-2n}\dd\xi\gg \int |\hat{\nu}(\xi)|^2|P_{1/\delta_2}*P_{1/\delta_2}(\xi)|^2\delta^{-2n}\dd\xi\\
	&=\delta^{-2n}\int|\nu*\widehat{P}^2_{1/\delta_2}(x)|^2\dd x.
	\end{split}
	\end{equation}
	By \eqref{equ:pos}, we have $\widehat{P}^2_{1/\delta_2}(x)=\widehat{P}^2_1(x/\delta_2)\gg \BB_{B(0,\delta_2)}(x)$. Combined with $|B(0,\delta_1)|\gg_C \delta^{-n}$, this implies
	$$\nu*\widehat{P}^2_{1/\delta_2}(x)\delta^{-n}\geq\nu*P_{\delta_1}(x)= \nu_{\delta_1}(x).$$ 
	Together with \eqref{equ:bo2}, we have the other direction of \eqref{equ:nudel}.
	
	The second inequality \eqref{equ:nudelcon} follows from the same argument. By Parseval's formula
	\begin{equation}\label{equ:nu*nu}
	\begin{split}
	\int \|\nu_{\delta_1}*(m_y)_*\nu_{\delta_1}\|_2^2\dd\nu(y)&=\int\int |\hat{\nu}_{\delta_1}(\xi)|^2|\widehat{(m_y)_*\nu_{\delta_1}}(\xi)|^2\dd\xi\dd\nu(y)\\
	&=\int\int|\hat\nu(\xi)|^2|\hat{\nu}(y\xi)|^2|\widehat{P}_{\delta_1}(y)^2||\widehat{P}_{\delta_1}(y\xi)|^2\dd\nu(y)\dd\xi\\
	&=\int\int|\hat\nu(\xi)|^2|\hat{\nu}(y\xi)|^2|\widehat{P}_1(\delta_1 y)^2||\widehat{P}_1(\delta_1 y\xi)|^2\dd\nu(y)\dd\xi.
	\end{split}
	\end{equation}
	For $y\in B(0,C)$ and $\xi\in B(0,2\delta^{-1})$, we have $\|\delta_1 y\xi\|\leq 1$. By \eqref{equ:pos}, the inequality \eqref{equ:nu*nu} implies \eqref{equ:nudelcon}.
\end{proof}

\begin{proof}[Proof of Lemma \ref{lem:alpbet}]
	Let $R=\delta^{-1}$. Consider $H_{R,t}=\{\xi\in B(0,R)||\hat{\nu}(\xi)|\geq t \}$, where $0<t<1$ will be fixed later. Since $\nu$ is supported on $B(0,K)$, the function $|\hat{\nu}|$ is $K$ Lipschitz. We have
	\[H_{R,t}+B\left(0,\frac{t}{2K}\right)\subset H_{R+1,\frac{t}{2}}. \]
	Hence by \eqref{equ:neiequ}
	$$\cal N_t(H_{R,t})\ll_K|H_{R,t}^{(\frac{t}{2K})}|t^{-n}\leq|H_{R+1,\frac{t}{2}}|t^{-n} .$$
	By the definition of $H_{R,t}$, Chebyshev's inequality and \eqref{equ:nvfoudec},
	\begin{equation}\label{equ:nthrt}
	\cal N_t(H_{R,t})\ll t^{-n-1}\int_{B(0,R+1)}|\hat{\nu}(\xi)|\dd\xi\ll R^{\beta}t^{-n-1}.
	\end{equation}
	
	From now on, suppose that $\|\xi\|\in[R/2,R]$. Let $H_{R,t}^\xi=\{x\in\bb R^n|x\xi\in H_{R,t} \}$. Then by $\|\xi\|\leq R$, we have $\|x\xi\|\leq R$ for $x\in\supp\mu\subset B(0,1)$, and 
	\begin{equation}\label{equ:nuxxi}
	\int|\hat{\nu}(x\xi)|\dd\mu(x)\leq t+\mu(H_{R,t}^\xi).
	\end{equation}
	We cover $H_{R,t}$ with balls of radius $t$ and we also get a cover of $H_{R,t}^\xi$ by $B^\xi(y,t)=\{x\in\RR^n|x\xi\in B(y,t) \}$. Due to $\|\xi\|\geq R/2$, there is at least one $j\in\{1,\dots,n \}$ such that $|\xi_j|\geq R/(2n)$. Therefore, we can replace $B^\xi(y,t)$ by a cylinder $\pi_j^{-1}B_\R(y,2n/R)$ and we obtain
	\[\mu B^\xi(y,t)=\mu\{x\in \RR^n|x\xi\in B(y,t) \}\ll\sup_{y\in\bb R,j=1,\dots, n}(\pi_j)_*\mu\{x|x\in B_\R(y,2n/R) \}. \]
	The above inequality combined with the hypothesis \eqref{equ:supmua} implies 
	\begin{equation}\label{equ:mubyt}
	\mu B^\xi(y,t)\ll R^{-\alpha}.
	\end{equation}
	Therefore by \eqref{equ:nthrt} and \eqref{equ:mubyt}
	\begin{equation}\label{equ:hrtxi}
	\mu(H_{R,t}^\xi)\leq \cal N_t(H_{R,t})\max\mu B^\xi(y,t) \ll_K R^{\beta-\alpha} t^{-(n+1)}.
	\end{equation} 
	If we take $t=R^{-\frac{\alpha-\beta}{n+2}}$, then the result follows from \eqref{equ:nuxxi} and \eqref{equ:hrtxi}.
\end{proof}
\section{Appendix}
The main purpose of the Appendix is to give a version of Theorem \ref{thm:sumfourier} (Proposition \ref{prop:sumproduct}) for its application in \cite{li2018spectralgap} to the products of random matrices.

In the application, we need to vary the measure. Using the same idea as in \cite[Propostion 3.2]{bourgain2017fourier}, we have a version for several different measures (Proposition \ref{prop:dissum}). The measures appearing in the random product of matrices are not compactly supported, hence we will relax the assumption on support in Proposition \ref{prop:sumproduct}.
\begin{prop}\label{prop:dissum1}
	Fix $\kappa>0$. Then there exist $k\in\bb N, \epsilon>0$ depending only on $\kappa$ such that the following holds for $\fren$ large enough. Let $\lambda$ be a Borel probability measure on $[\frac 1 2,1]^n\subset\RR^n$. Assume that for all $\rho\in[\fren^{-1},\fren^{-\epsilon}]$
		\begin{equation}
		\quad\sup_{a\in\bb R,v\in\bb S^{n-1}}(\pi_v)_*\lambda(B_\RR(a,\rho))=\sup_{a,v}\lambda\{x| \l v, x\r\in B_\RR(a,\rho) \}\leq \rho^{\kappa}.
		\end{equation}	
	Then for $\xi$ in $\RR^n$ with $\|\xi\|\in [\fren/2,\fren]$
	\[ \left|\int\exp(2i\pi\l  \xi, x_1\cdots x_k\r)\dd\lambda(x_1)\dots \dd\lambda(x_k)\right|\leq \fren^{-\epsilon}. \]
\end{prop}
\begin{proof}
	By Lemma \ref{lem:equivalent}, Theorem \ref{thm:sumfourier} implies the result.
\end{proof}	
We state a version with different measures.
\begin{prop}\label{prop:dissum}
	Fix $\kappa>0$. Then there exist $k\in\bb N, \epsilon>0$ depending only on $\kappa$ such that the following holds for $\fren$ large enough. Let $\lambda_1,\dots \lambda_k$ be Borel measures on $[\frac 1 2,1]^n\subset\RR^n$ with total mass less than 1. Assume that for all $\rho\in[\fren^{-1},\fren^{-\epsilon}]$ and $j=1,\dots,k$
	\begin{equation}
	\quad\sup_{a\in\bb R,v\in\bb S^{n-1}}(\pi_v)_*\lambda_j(B_\RR(a,\rho))=\sup_{a,v}\lambda_j\{x| \l v, x\r\in B_\RR(a,\rho) \}\leq \rho^{\kappa}.
	\end{equation}	
	Then for $\xi$ in $\RR^n$ with $\|\xi\|\in [\fren/2,\fren]$
	\[ \left|\int\exp(2i\pi\l  \xi, x_1\cdots x_k\r)\dd\lambda_1(x_1)\dots \dd\lambda_k(x_k)\right|\leq \fren^{-\epsilon}. \]
\end{prop}
\begin{proof}
	The proof is the same as the argument in \cite[Propostion 3.2]{bourgain2017fourier}. For completeness, we give a sketch here.
	
	We first verify that if the mass of the measure $\lambda$ is less than $1$, the result of Proposition \ref{prop:dissum1} also holds. Let $\epsilon_2$ be given by Proposition \ref{prop:dissum1} when the regular exponent equals $\kappa/2$. We distinguish two cases
	\begin{itemize}
		\item If $\lambda(\R^n)\geq \tau^{-\epsilon_2\kappa/2 }$, then replace $\lambda$ by $\lambda'=\lambda/\lambda(\R^n)$. For $\rho\in[\tau^{-1},\tau^{-\epsilon_2}]$, we have
		\[\sup_{a\in\bb R,v\in\bb S^{n-1}}(\pi_v)_*\lambda'(B_\RR(a,\rho))\leq \tau^{\epsilon_2\kappa/2}\sup_{a\in\bb R,v\in\bb S^{n-1}}(\pi_v)_*\lambda(B_\RR(a,\rho))\leq \tau^{\epsilon_2\kappa/2}\rho^\kappa. \]
		Due to $\rho\leq \tau^{-\epsilon_2}$, we have $\tau^{\epsilon_2\kappa/2}\rho^{\kappa}\leq \rho^{\kappa/2}$. The measure $\lambda'$ satisfies non concentration with $\kappa/2$. By Proposition \ref{prop:dissum1}, we have the result.
		\item If $\lambda(\R^n)< \tau^{-\epsilon_2\kappa/2 }$, then we have
		\[ \int\exp(2i\pi\l  \xi, x_1\cdots x_k\r)\dd\lambda(x_1)\dots \dd\lambda(x_k)|\leq \tau^{-k\epsilon_2\kappa/2 }. \]
	\end{itemize}
	Hence we can take $\epsilon=\min\{\epsilon_2,k\epsilon_2\kappa/2 \}$.
	
	Then we want to prove that the result holds for different measures. For $z\in \RR^k$, let $\lambda_z=\sum_{1\leq j\leq k}z_j\lambda_j$. Let
	\[G(\lambda_1,\cdots,\lambda_k)=\int\exp(2i\pi\l  \xi, x_1\cdots x_k\r)\dd\lambda_1(x_1)\dots \dd\lambda_k(x_k) \]
	and
	$$F(z)=F(z_1,\cdots, z_k)=G(\lambda_z,\cdots,\lambda_z).$$
	Then $F(z)$ is polynomial of $k$ variables of degree $k$, and $k!G(\lambda_1,\cdots,\lambda_k)$ is the coefficient of $z_1\cdots z_k$ in $F(z)$. For $z\in \RR_{\geq 0}^k$, we have
	\[|F(z)|\leq |z|^k\tau^{-\epsilon}, \text{ where }|z|=\sum_{1\leq j\leq k}|z_j|, \]
	by using the result of the first part with $\lambda=\frac{1}{|z|}\lambda_z$.
	\begin{lem}
		Let $F$ be a polynomial of $k$ variables of degree less than $n$. Let $h(F)$ be the maximum of the absolutely value of the coefficients in $F$. Then
		\[h(F)\leq O_{k,n}\sup_{z\in\{0,\cdots,n\}^k\subset\R^k }\{|F(z)| \} . \]
	\end{lem}
	In this lemma, we define two norms on the space of polynomials  of $k$ variable of degree less than $n$. The inequality is due to the equivalence of norms on finite dimensional vector space.
	Hence
	\[|G(\lambda_1,\cdots,\lambda_k)|\ll_k h(F)\ll_k \tau^{-\epsilon}. \]
	The proof is complete.
\end{proof}
Now we will give another version of Fourier decay of multiplicative convolution, which relaxes the assumption on the support of $\lambda_j$.
\begin{prop}
	\label{prop:sumproduct}
		Fix $\kappa_0>0$. Let $C_0>0$. Then there exist $\epsilon_2$ and $k\in\bb N$ depending only on $\kappa_0$ such that the following holds for $\fren$ large enough depending on $C_0,\kappa_0$.  Let $\lambda_1,\dots \lambda_k$ be Borel measures on $\RR^n$ supported in $([-\fren^{\epsilon_3},-\fren^{-\epsilon_3}]\cup[\fren^{-\epsilon_3},\fren^{\epsilon_3}])^n$ with total mass less than $1$, where $\epsilon_3=\min\{\epsilon_2,\epsilon_2\kappa_0,1\}/10k$. Assume that for all $\rho\in[\fren^{-2},\fren^{-\epsilon_2}]$ and $j=1,\dots,k$
		\begin{equation}
		\quad\sup_{a\in\bb R,v\in\bb S^{n-1}}(\pi_v)_*\lambda_j(B_\RR(a,\rho))=\sup_{a,v}\lambda_j\{x| \l v, x\r\in B_\RR(a,\rho) \}\leq C_0\rho^{\kappa_0}.
		\end{equation}	
		
	Then for all $\varsigma\in\RR^n, \|\varsigma\|\in[\fren^{3/4},\fren^{5/4} ]$ we have
	\begin{equation}
	\left|\int\exp(2i\pi\l \varsigma, x_1\cdots x_k\r)\dd\lambda_1(x_1)\cdots\dd\lambda_k(x_k) \right|\leq \fren^{-\epsilon_2}.
	\end{equation}
\end{prop}
\begin{rem}
	The proof is tedious, but the idea is clear. If the non concentration assumption is valid in some large range, then there is some place to rescale a little the measure and the result still holds. We only need to find some exponent $\epsilon_3$ carefully. 
\end{rem}
\begin{proof}
	It is sufficient to prove the case that $\supp\lambda_j\in [\fren^{-\epsilon_3},\fren^{\epsilon_3}]^n$. In face, we can divide each measure into $\lambda_j=\sum_{m\in (\bb Z/2\bb Z)^n}\lambda_j^m$, where $\lambda_j^m$ is the unique part of $\lambda_j$ whose support is in the same orthant as $m$ and we identify $(\bb Z/2\bb Z)^n$ with $\{-1,1\}^n\in \RR^n$. Then 
	\begin{align*}
		&\ \left|\int\exp(2i\pi\l \varsigma, x_1\cdots x_k\r)\dd\lambda_1^{m_1}(x_1)\cdots\dd\lambda_k^{m_k}(x_k) \right|
		\\&=\left|\int\exp(2i\pi(\l \varsigma m_1\cdots m_k, x_1\cdots x_k\r)\dd\lambda_1^{m_1}(m_1x_1)\cdots\dd\lambda_k^{m_k}(m_kx_k) \right|.
	\end{align*}
	We know that the support of measure $(m_j)_*\lambda_j$ is in $[\fren^{-\epsilon_3},\fren^{\epsilon_3}]^n$. Hence by the result of the case $\supp\lambda_j$ positive, we have the result with a constant $2^{nk}$.
	
	Let $\epsilon$ as in Proposition \ref{prop:dissum} with $\kappa=\kappa_0/2$, and let $\epsilon_2=\epsilon/4$.
	
	Divide $[\fren^{\epsilon_3},\fren^{-\epsilon_3}]^n$ into $[2^l,2^{l+1}]:=[2^{l_1},2^{l_1+1}]\times\cdots [2^{l_n},2^{l_n+1}]$ with $l\in \bb Z^n$. We rescale the measure in each interval to $[1/2,1]^n$. Let $\lambda_j^l(A)=\lambda_j|_{[2^{l-1},2^l]}(2^{l}A)$.  For $\rho\in[\fren^{-3/2},\fren^{-\epsilon_2/2} ]$ we have 
	\begin{align}\label{equ:piv}
	(\pi_v)_*\lambda_j^l(B_\RR(a,\rho))\leq	(\pi_{v'})_*\lambda_j(\|2^{-l}v\|^{-1}B_\RR(a,\rho)),
	\end{align}
	where $v'=2^{-l}v/\|2^{-l}v\|$.
	The inequality $\|2^{-l}v\|\in[\fren^{-\epsilon_3},\fren^{\epsilon_3}]$ implies that we have $\|2^{-l}v\|^{-1}\rho\in [\fren^{-3/2-\epsilon_3},\fren^{-\epsilon_2/2+\epsilon_3}]\subset[\fren^{-2},\fren^{-\epsilon_2/4}]$. Due to $\rho^{-1/4}\geq \fren^{\epsilon_2/8}\geq \fren^{\epsilon_3}\geq \|2^{-l}v\|^{-1}$ for $\rho\in[\fren^{-3/2},\fren^{-\epsilon_2/2} ]$, by \eqref{equ:piv} we have
	\begin{equation}\label{equ:dimml}
	(\pi_v)_*\lambda_j^l(B_\RR(a,\rho))\leq C_0 (\|2^{-l}v\|^{-1}\rho)^{\kappa_0} \leq \rho^{\kappa_0/2},
	\end{equation}
	for $\fren$ large enough depending on $C_0$.
	
	Summing up over $\|l\|\leq \epsilon_3\log_2\fren$, we have
	\begin{align*}
	&\left|\int\exp(2i\pi\l \varsigma, x_1\cdots x_k\r)\dd\lambda_1(x_1)\cdots\dd\lambda_k(x_k) \right|\\
	&\leq
	\sum_{l^j\in\bb Z^n,\|l^j\|\leq \epsilon_3\log\tau}\left|\int\exp(2i\pi\l \varsigma, x_1\cdots x_k\r)\dd\lambda_1^{l^1}(2^{-l^1}x_1)\cdots\dd\lambda_k^{l^k}(2^{-l^k}x_k) \right|
	\\
	&= \sum_{l^j\in\bb Z^n,\|l^j\|\leq \epsilon_3\log\tau}\left|\int\exp(2i\pi\l \varsigma, 2^{l^1+\cdots+l^k} y_1\cdots y_k\r)\dd\lambda_1^{l^1}(y_1)\cdots\dd\lambda_k^{l^k}(y_k) \right|.
	\end{align*}
	Let $\fren_1=\|\varsigma 2^{l^1+\cdots+l^k}\|$, then $\fren_1\in[\fren^{3/4-k\epsilon_3},\fren^{5/4+k\epsilon_3}]$. Then we have $[\fren_1^{-1},\fren_1^{-\epsilon_2}]\subset[\fren^{-3/2},\fren^{-\epsilon_2/2}]$. The assumption of Proposition \ref{prop:dissum} is verified by \eqref{equ:dimml} with $\fren$ replaced by $\fren_1$. Therefore 
	\begin{align*}
	&\sum_{l^j}\left|\int\exp(2i\pi\l \varsigma, 2^{l^1+\cdots+l^k} y_1\cdots y_k\r)\dd\lambda_1^{l^1}(y_1)\cdots\dd\lambda_k^{l^k}(y_k) \right|\leq \sum_{l^j\in\bb Z^n,\|l^j\|\leq \epsilon_3\log\tau}\|\varsigma 2^{l^1+\cdots+l^k}\|^{-\epsilon_2}\\
	&\leq (2\epsilon_3\log_2\tau)^{kn}(\tau^{3/4-k\epsilon_3})^{-\epsilon}\leq \tau^{-\epsilon/4},
	\end{align*}
	when $\tau$ is large enough depending on $k,n,\epsilon$. The proof is complete.
\end{proof}
\subsection{The case of $\C$}
In this part, we consider the case of $\C$. The decay of Fourier transform of multiplicative convolution of measures was already indicated in \cite{bourgain2017fourier}. They said that by using the sum-product estimate on $\C$ \cite{bourgain_spectral_2012} and the same method as the real case we can obtain the complex case. Here we indicate that our approach can also give the result for $\C$. This result will be used in a joint work with Frédéric Naud and Wenyu Pan \cite{LNP} on the Fourier decay of Patterson-Sullivan measures associated to Kleinian Schottky groups.

Let $\mu$ be a Borel probability measure on $\C$. For $\xi$ in $\C$, we define the Fourier transform of $\mu$ by
\[\hat{\mu}(\xi):=\int e^{2i\pi\Re(\xi z)}\dd\mu(z). \]
As two algebras of dimension 2, the algebras $\R^2$ and $\C$ are different in their product structure.

We need a version of discretized sum-product  estimate on $\C$.
\begin{prop}\label{prop:sumproC}
	Given $\kappa>0,\sigma\in(0,2)$, there exists $\epsilon>0$ such that for all $\delta>0$ sufficiently small, if $A, X\subset B_\C(0,\delta^{-\epsilon})$ satisfy the following $(\delta,\kappa,\sigma,\epsilon)$ assumption:
	
	(i)
	\begin{equation*}
	\forall\rho\geq\delta,\ \nrho(A)\geq\delta^\epsilon\rho^{-\kappa},
	\end{equation*} 
	
	(ii) $A$ is $\delta^\epsilon$ away from $\R$ in $\C$,
	
	(iii)
	\begin{equation*}
	\forall\rho\geq\delta,\ \nrho(X)\geq\delta^\epsilon\rho^{-\kappa},
	\end{equation*}
	
	(iv) $\ndelta(X)\leq \delta^{-(n-\sigma)-\epsilon}$.
	
	Then
	\begin{equation*}
	\ndelta(X+X)+\sup_{a\in A}\ndelta(X+aX)\geq \delta^{-\epsilon}\ndelta(X).
	\end{equation*}
\end{prop}
This is a consequence of \cite[Theorem 3]{he_discretized_2016} (see also \cite[Proposition 2]{bourgain_spectral_2012}, where a constant is needed instead of $\delta^\epsilon$ in the assumption (ii)). 
\begin{prop}[He]\label{prop:sumprolin}
	Given $\kappa>0,\sigma\in(0,2)$, there exists $\epsilon>0$ such that for all $\delta>0$ sufficiently small, if $X\subset B_{\R^2}(0,\delta^{-\epsilon})$ and $A\subset B_{End(\R^2)}(0,\delta^{-\epsilon})$ satisfy the following $(\delta,\kappa,\sigma,\epsilon)$ assumption:
	
	(i)
	\begin{equation*}
	\forall\rho\geq\delta,\ \nrho(A)\geq\delta^\epsilon\rho^{-\kappa},
	\end{equation*} 
	
	(ii) for every nonzero proper linear subspaces $W\subset \R^2$ , there is $a \in A$
	and $w \in W \cap B_{\R^2}(0, 1)$ such that $d(aw, W ) \geq \delta^\epsilon$ .
	
	(iii)
	\begin{equation*}
	\forall\rho\geq\delta,\ \nrho(X)\geq\delta^\epsilon\rho^{-\kappa},
	\end{equation*}
	
	(iv)$\ndelta(X)\leq \delta^{-(n-\sigma)-\epsilon}$.
	
	Then
	\begin{equation*}
	\ndelta(X+X)+\sup_{a\in A}\ndelta(X+aX)\geq \delta^{-\epsilon}\ndelta(X).
	\end{equation*}
\end{prop}
\begin{proof}[Form Proposition \ref{prop:sumprolin} to Proposition \ref{prop:sumproC}]
We identify $\C$ with $\R^2$, and the multiplication of a complex number $x+iy$ for $x,y\in\R$ is seen as the multiplication of the matrix $\begin{pmatrix}
x & -y \\ y & x
\end{pmatrix}$. Then we only need to verify Assumption (ii) of Proposition \ref{prop:sumprolin}. By Assumption (ii) of Proposition \ref{prop:sumproC}, there exists $a\in A$ such that $d(a,\R)\geq \delta^{\epsilon}$. We take $w$ in $W$ with unit length, since the distance is invariant under the rotation, then
\[d(aw,W)=d(a,\R)\geq\delta^{\epsilon}. \]
The proof is complete.
\end{proof}

With the same argument as in Section \ref{sec:appmul}, by using Proposition \ref{prop:sumproC} we have
\begin{prop}
		Given $\kappa_0>0$, there exist $\epsilon,\epsilon_1>0$ and $k\in \bb N$ such that the following holds for $\delta>0$ small enough. Let $\mu$ be a probability measure on the annulus $\{z\in \mathbb{C}: 1/2\leq |z|\leq 2\}$ which satisfies $(\delta,\kappa_0,\epsilon)$ projective non concentration assumption, that is, 
		\begin{equation}
		\forall\rho\geq\delta,\quad\sup_{a\in\bb R,\theta\in \R}\mu\{z\in\C| \Re(e^{i\theta}z)  \in B_\R(a,\rho) \}\leq\delta^{-\epsilon} \rho^{\kappa_0}.
		\end{equation}	
		
		Then for all $\xi\in\C$ with $\|\xi\|\in[\delta^{-1}/2,\delta^{-1}]$,
		\begin{equation}
		|\hat{\mu}_k(\xi)|=\left|\int \exp(2i\pi\Re(\xi z_1\cdots z_k))\dd\mu(z_1)\cdots \dd\mu(z_k)\right|\leq \delta^{\epsilon_1}.
		\end{equation}
\end{prop}

Then by the same argument as in \cite[Section 3.1]{bourgain2017fourier} and Lemma \ref{lem:equivalent}, we obtain
\begin{prop}
	Given $\kappa_0>0$, there exist $\epsilon>0$ and $k\in \bb N$ such that the following holds for $\delta\in(0,1)$. Let $C_0>0$ and let $\mu_1,\cdots, \mu_k$ be Borel measures on the annulus $\{z\in \mathbb{C}: 1/C_0\leq |z|\leq C_0\}$ with total mass less than $C_0$ and which satisfy projective non concentration assumption, that is, for $j=1,\cdots ,k$
	\begin{equation}
	\forall\rho\in[C_0\delta,C_0^{-1}\delta^{\epsilon}],\quad\sup_{a\in\bb R,\theta\in \R}\mu_j\{z\in\C| \Re(e^{i\theta}z)  \in B_\R(a,\rho) \}\leq C_0\rho^{\kappa_0}.
	\end{equation}	
	
	Then there exists a constant $C_1$ depending only on $C_0,\kappa_0$ such that for all $\xi\in\C$ with $\|\xi\|\in[\delta^{-1}/2,\delta^{-1}]$,
	\begin{equation}
	\left|\int \exp(2i\pi\Re(\xi z_1\cdots z_k))\dd\mu_1(z_1)\cdots \dd\mu_k(z_k)\right|\leq C_1\delta^{\epsilon}.
	\end{equation}
\end{prop}

\section*{Acknowledgement}
This is part of author's Ph.D. thesis, written under the supervision
of Jean-François Quint at the University of Bordeaux.
The author wishes to express his thanks to Weikun He and Nicolas de Saxcé for explaining to me their work. %?? 
The author would like to thank Jean-François Quint for his detailed comments
on an earlier version of the paper.

The author would also like to thank the referee for pointing out a mistake and for carefully reading the manuscript.

\noindent Jialun LI\\
Current address: Universit\"at Z\"urich\\
jialun.li@math.uzh.ch
\end{document}